\documentclass{amsart}

\usepackage{bm}
\usepackage[T1]{fontenc}
\usepackage{amssymb}
\usepackage{mathrsfs}
\usepackage{graphicx}

\usepackage{color}
\usepackage{mathtools}
\mathtoolsset{showonlyrefs=true}

\usepackage[utf8]{inputenc}

\usepackage{amsthm}

\newcommand{\K}{\mathscr{K}}
\newcommand{\wa}{w_\alpha}
\newcommand{\Bo}{\mathscr{B}}
\newcommand{\Bop}{\dot{\mathscr{B}}}
\newcommand{\Boc}{\mathscr{B}_c}
\newcommand{\Ci}{\mathscr{C}}
\newcommand{\N}{\mathbb{N}}
\newcommand{\R}{\mathbb{R}}
\newcommand{\C}{\mathbb{C}}
\newcommand{\Z}{\mathbb{Z}}
\newcommand{\D}{\mathbb{D}}

\newcommand{\Hu}{\mathbb{R}_+^{n+1}}

\newcommand{\De}{\mathscr{D}}
\newcommand{\s}{\mathscr{S}}

\newcommand{\F}{\mathcal{F}}

\newcommand{\supp}{\operatorname{supp}}
\newcommand{\ve}{\varepsilon}
\newcommand{\Es}{\mathscr{S}}
\newcommand{\hdot}{\bm{\ldotp}}

\newcommand{\re}{\operatorname{Re}}

\newcommand{\divv}{\operatorname{div}}

\theoremstyle{plain}
\newtheorem{prop}{Proposition}[section]
\newtheorem{thm}[prop]{Theorem}
\newtheorem{lemma}[prop]{Lemma}
\newtheorem{cor}[prop]{Corollary}

\theoremstyle{definition}
\newtheorem{dfn}[prop]{Definition}

\theoremstyle{remark}

\newtheorem*{rmk}{Remark}

\numberwithin{equation}{section}

\begin{document}

\title[Generalized axially symmetric potentials]{Generalized axially symmetric potentials with
distributional boundary values}

\date{\today}

\author{Jens Wittsten}
\address{Graduate School of Human \& Environmental Studies,
Kyoto University,
Yoshida Nihonmatsu-cho, Sakyo-ku Kyoto 606-8501, Japan
}
\email{jensw@maths.lth.se}

\begin{abstract}
We study a counterpart of the classical Poisson integral for a family of weighted Laplace differential
equations in Euclidean half space, solutions of which are known as generalized axially symmetric potentials.
These potentials appear naturally in the study of hyperbolic Brownian motion with drift.
We determine the optimal class of tempered distributions which by means of the so-called $\s'$-convolution
can be extended to generalized axially symmetric potentials.
In the process, the associated Dirichlet boundary value problem is solved,
and we obtain sharp order relations for the asymptotic growth of these extensions.
\end{abstract}

\keywords{Generalized axially symmetric potential, Poisson integral, weighted Laplace operator, Poisson kernel, 
weighted space of distributions, hyperbolic Brownian motion}

\subjclass[2010]{Primary: 31B05; Secondary: 35J25}

\maketitle

\section{Introduction}

Consider in $n+1$ dimensions the elliptic partial differential equation
\begin{equation}\label{Dalfa}
D_\alpha u\equiv y^{-\alpha}\bigg( \frac{\partial^2 u}{\partial x_1^2}+\cdots+\frac{\partial^2 u}{\partial x_n^2}+\frac{\partial^2 u}{\partial y^2}
-\frac{\alpha}{y}\frac{\partial u}{\partial y}\bigg)=0,
\end{equation}
where $\alpha$ is an arbitrary real parameter.
When $\alpha=0$ we recover the classical Laplace equation, and when $\alpha$ is a negative integer and $n=1$
then \eqref{Dalfa} is satisfied by the family of axially symmetric harmonic functions in
$(2-\alpha)$-dimensional space,
considered in a meridian plane. Solutions to \eqref{Dalfa} have therefore historically been referred to as
generalized axially symmetric potentials, see the exposition by Weinstein~\cite{Weinstein}.
This theory proved to be a very strong tool allowing treatment
of various problems in for example fluid mechanics and generalized Tricomi equations~\cite{Weinstein,Weinstein2}.
In this context, the operator $y^{\alpha+1}D_\alpha$ has traditionally been denoted by $L_{k}$
with parameter $k=-\alpha$, that is,
\[
L_k u\equiv y\bigg( \frac{\partial^2 u}{\partial x_1^2}+\cdots+\frac{\partial^2 u}{\partial x_n^2}+\frac{\partial^2 u}{\partial y^2}
\bigg)+k\frac{\partial u}{\partial y},\quad k\in\R.
\]

The equation $D_\alpha u=0$ is the Laplace-Beltrami equation in the Riemannian space defined by the metric
\[
ds^2=y^{2\alpha/(1-n)}\bigg( \sum_{i=1}^n dx_i^2+dy^2\bigg),\quad n>1.
\]
This fact has recently led to the appearance of the operator $D_\alpha$ in connection with the study of so-called
$(n+1)$-dimensional hyperbolic Brownian motion with drift.
This area has been of much interest lately, since it is related to geometric Brownian motion
and Bessel processes and has applications to risk theory, see
the recent work by Ma{\l}ecki and Serafin~\cite{MaleckiSerafin} and the
references therein. Contributions have also been made by among others
Baldi, Byczkowski, Casadio Tarabusi, Fig{\`a}-Talamanca, Graczyk, Ryznar, St{\'o}s, and Yor
\cite{BaldiEtAl2,BaldiEtAl,ByczkowskiEtAl,ByczkowskiRyznar}.
We recall that if $\mathbb H^{n+1}$ denotes the half space model of $(n+1)$-dimensional real hyperbolic space,
that is, $\mathbb H^{n+1}$ is the space $\{(x,y)\in\R^{n+1}:y>0\}$ endowed with the metric $ds^2=y^{-2}(dx_1^2+\cdots dx_n^2+dy^2)$,
then an $(n+1)$-dimensional hyperbolic Brownian motion with drift in $\mathbb H^{n+1}$
is defined as a diffusion corresponding to the system of stochastic differential equations
\begin{equation}\label{HyperbolicBrownianMotion}
\begin{cases}
dX_t &  \!\!\!=  Y_tdW_t,\\
dY_t &  \!\!\!=  Y_tdB_t-(\mu-\frac{1}{2}) Y_t dt.
\end{cases}
\end{equation}
Here $W_t$ and $B_t$ are independent $n$-dimensional and one-dimensional Brownian motions, respectively. 
The generator of this diffusion is given by
\[
y^2\bigg(\sum_{i=1}^{n}\partial_{x_i}^2+\partial_{y}^2\bigg)-(2\mu-1)y\partial_y
\]
divided by a factor 2, which is a non-constant multiple of $D_\alpha$ for the parameter value $\alpha=2\mu-1$.
The case $\mu=n/2$ corresponds to classical hyperbolic Brownian motion on $\mathbb H^{n+1}$.

In this paper we shall study a half space boundary value problem for the operator $D_\alpha$.
Let $\Hu=\{(x,y)\in\R^n\times\R:y>0\}$ denote the half space in $n+1$ dimensions.
We will identify the boundary of $\Hu$ with $\R^n$. Introduce the weight function
$\gamma_\alpha(x,y)=y^{-\alpha}$ for $(x,y)\in\Hu$, and note that $D_\alpha$ can be written in
divergence form
\begin{equation}\label{PDOdiv}
D_\alpha u=\divv{(\gamma_\alpha\nabla u)},\quad\gamma_\alpha(x,y)=y^{-\alpha},
\quad (x,y)\in\Hu.
\end{equation}
Since the behavior of solutions to $D_\alpha u=0$ are essentially different for the two
parameter ranges $\alpha>-1$ and $\alpha\le-1$ (see the result due to Huber~\cite{HuberUnique}, included below
as Theorem \ref{thm:Huberuniqueness})
we shall restrict our attention to the parameter range $\alpha>-1$ and study the Dirichlet problem 
\begin{equation}\label{Dirichletproblem}
\left\{
\begin{aligned}
D_\alpha u & =0 \quad\text{in}\ \Hu, \\
u & =f  \quad\text{in}\ \R^n.
\end{aligned}
\right.
\end{equation}
We will assume that the boundary data $f\in\Es'$ is a tempered distribution in $\R^n$;
the boundary condition is to be understood
as $\lim_{y\to0}u_y=f$ in $\Es'$, where 
\begin{equation}\label{uyfcn}
u_y(x)=u(x,y),\quad x\in\R^n,
\end{equation}
for $y>0$.

Due to the definition of the weight $\gamma_\alpha$, \eqref{Dirichletproblem} is singular on the hyperplane $y=0$.
We mention here that $D_\alpha$
is related to a certain weighted complex Laplace operator $\Delta_\alpha$ in the unit disc $\D$,
which exhibits similar behavior near the boundary of $\D$.
In fact, when $n=1$ the operator $D_\alpha$ is realized as (a multiple of) the symmetric part
of the differential operator
\[
u(z)\mapsto\partial_z \gamma_\alpha(y)\bar\partial_z u(z),\quad z=x+iy\in\C,\quad y>0,
\]
where $\partial$ and $\bar\partial$ are the usual complex derivatives. The mentioned weighted Laplace operator
$\Delta_\alpha$ is obtained by replacing the weight function $\gamma_\alpha$
by its counterpart in the unit disc, the so-called standard weight $z\mapsto (1-|z|^2)^{-\alpha}$
for $z\in\D$, appearing in connection to Bergman space theory. The corresponding Dirichlet problem
for $\Delta_\alpha$ in the unit disc was recently solved by the author in collaboration with A. Olofsson~\cite{OWjapanpaper}.
We also mention the recent paper by Olofsson~\cite{Olofsson2012differential} which in a certain
sense studies the unit disc analog of the Dirichlet problem \eqref{Dirichletproblem}.
The family of operators studied by Olofsson~\cite{Olofsson2012differential}
has been shown to be connected to weighted integrability of polyharmonic functions in the unit disc
by Borichev and Hedenmalm~\cite{BorichevHedenmalm}. 
See also Hedenmalm~\cite{Hedenmalm2014}.

The singular or degenerate behavior of $D_\alpha$ near the boundary means that
the theory for strictly elliptic partial differential equations
is not applicable to \eqref{Dirichletproblem}. This notwithstanding, much is still known about the existence and uniqueness of solutions to the
Dirichlet problem \eqref{Dirichletproblem} when the data is regular.
In particular, the notion corresponding to a Poisson integral appears in Weinstein~\cite{Weinstein}
in the case $n=1$, using a kernel function corresponding to
\begin{equation*}
\K_\alpha(x,y)= \frac{\Gamma((\alpha+n+1)/2)}{\Gamma((\alpha+1)/2)\pi^{n/2}}\cdot
\frac{y^{\alpha+1}}{(x^2+y^2)^{(\alpha+n+1)/2}},
\quad (x,y)\in\Hu,
\end{equation*}
where $\Gamma(s)$ is the Gamma function. We study properties of this kernel in Section \ref{sec:GASP}.
(In the context of hyperbolic Brownian motion, the term {\it probability
density function} is commonly used.)
Moreover, a fundamental solution in any dimension was found by Diaz and Weinstein~\cite{DiazWeinstein},
and a generalized Poisson kernel for the Dirichlet problem in a hemisphere was also provided by Huber~\cite{HuberUnique}.
Similar formulas appear in the more recent paper by Caffarelli and Silvestre~\cite{CaffarelliSilvestre} (and
in other subsequent work) on the fractional Laplacian $(-\Delta)^\mu$, $0<\mu<1$, where 
$(-\Delta)^\mu$ was shown to be related to the extension problem \eqref{Dirichletproblem} for $\alpha=2\mu-1$ through the Dirichlet to Neumann map.

However, for boundary data $f\in\Es'$, it is not immediately clear how to define
the ``Poisson integral of $f$'' by means of the kernel function $\K_\alpha$. The natural choice would be through convolution
$\K_{\alpha,y}\ast f$ in the sense of distributions,
with $\K_{\alpha,y}$ interpreted in accordance with \eqref{uyfcn}, but this is not applicable in our
case since the Fourier transform of $\K_{\alpha,y}$ is not smooth at the
origin, see Theorem \ref{thm:Ftransform}. (Recall that the convolution
$u\ast v$ was originally defined by Schwartz~\cite{Schwartz}
for pairs $u\in\mathscr O_C'$ and $v\in\Es'$, where $\mathscr O_C'$ is the space of rapidly decaying distributions,
and that the Fourier transform is an isomorphism between $\mathscr O_C'$ and 
the space $\mathscr O_M$ of slowly growing smooth functions, see Schwartz~\cite[Chapitre VII]{Schwartz}.)
To circumvent this problem we shall use the so-called $\Es'$-convolution
proposed by Hirata and Ogata~\cite{HirataOgata} and later given an equivalent form by Shiraishi~\cite{Shiraishi}.

In Section \ref{section:Weighteddistributions}
we recall the definitions of certain weighted spaces of distributions
(continuously embedded in $\Es'$), and determine the optimal class 
of tempered distributions $f$ for which
the $\Es'$-convolution $\K_{\alpha,y}\ast f$ is well defined for all $y>0$,
see Theorem \ref{thm:convolutionwelldefined}. 
In Section \ref{sec:Dirichletproblem}
we define the Poisson integral $\mathscr \K_\alpha\lbrack f\rbrack :(x,y)\mapsto \K_{\alpha,y}\ast f(x)$
for $f$ in this class, 
and show that it has boundary limit $f$ in $\Es'$, see Theorem \ref{thm:existence}.
We also establish that
$u = \mathscr \K_\alpha\lbrack f\rbrack$ solves $D_\alpha u=0$ in $\Hu$ for such $f$,
thus proving existence of solutions to the Dirichlet problem \eqref{Dirichletproblem},
see Corollary \ref{cor:existence}.
A similar approach has been used by Alvarez, Guzm{\'a}n-Partida and Sk{\'o}rnik
\cite{Alvarezetal1} to characterize the tempered distributions that are $\Es'$-convolvable with the classical Poisson kernel
$\K_0$ for the half space, and further used by Alvarez, Guzm{\'a}n-Partida and P{\'e}rez-Esteva
\cite{Alvarezetal2} to study harmonic extensions of distributions. To prove our results we adapt the ideas found
in the mentioned papers to the full parameter range $\alpha>-1$.
At the end of Section \ref{sec:Dirichletproblem} we also calculate
the kernel function for the Dirichlet problem $D_\alpha u=0$ in the half space
$y>\eta$ where $\eta>0$,
see Proposition \ref{prop:hitting}.
This kernel is the density of the {\it hitting distribution}
appearing in hyperbolic Brownian motion.

In Section \ref{sec:asymptoticgrowth}
we study asymptotic behavior of the
Poisson integral $\K_\alpha\lbrack f\rbrack$.
By using methods similar to
those used by Siegel and Talvila~\cite{SiegelTalvila} to study the classical Poisson integral
in the harmonic case $\alpha=0$,
the order relations for the asymptotic growth that we obtain (Theorem \ref{thm:growthestimate}) are shown to be sharp.
For comparison, we also consider the issue of uniqueness of solutions:
As evidenced by the function $u(x,y)=y^{\alpha+1}$,
which solves $D_\alpha u=0$ in $\Hu$ and vanishes at the boundary if $\alpha>-1$,
solutions to \eqref{Dirichletproblem} are not unique in general unless additional restrictions of growth
at infinity are imposed. We include one result in this direction, proved by using
a Phragm{\'e}n-Lindel{\"o}f principle due to Huber~\cite{HuberPL}
together with a regularization argument,
see Corollary \ref{Dalfauniqueness.distribution}.
However, the growth conditions imposed are
not compatible with the established asymptotic behavior of the
Poisson integral $\K_\alpha\lbrack f\rbrack$ for general $f$, so this does unfortunately not lead to
a satisfactory representation theory.

We finally wish to mention that the operator $D_\alpha$ superficially resembles
the governing operator in what is known as Calder{\'o}n's inverse conductivity problem; however,
the conditions on the weight function are totally different. In fact,
let $\Omega$ be a bounded domain in $\R^n$ for $n\ge2$, and let $\gamma$ be
a real-valued function in $L^\infty(\R^n)$ with a positive lower bound. Consider
the conductivity equation
\[
\left\{
\begin{aligned}
\divv{(\gamma\nabla u)} & =0 &\text{in}\ \Omega, \\
u & =f  &\text{on}\ \partial\Omega.
\end{aligned}
\right.
\]
In 1980, Calder{\'o}n~\cite{Calderon} posed the question whether the conductivity $\gamma$
could be recovered from the boundary measurements as described by the Dirichlet to Neumann map $\Lambda_\gamma$.
This problem, known in medical imaging as Electrical Impedance Tomography,
has been intensely studied and numerous positive results are known under slightly
stronger regularity assumptions on $\gamma$.
In 2 dimensions, the problem was recently solved by Astala and P{\"a}iv{\"a}rinta~\cite{AstalaPaivarinta} who showed that
$\gamma\in L^\infty(\R^2)$ is completely determined by $\Lambda_\gamma$ even if the boundedness
assumption on the domain $\Omega$ is dropped. For more on this, we refer to the mentioned paper
and the references therein.

\section{The kernel function}\label{sec:GASP}

In this section we discuss the kernel function $\K_\alpha$ for the Dirichlet problem \eqref{Dirichletproblem}
mentioned in the introduction, and calculate its Fourier transform.
However, we first indicate the difference in behavior of solutions
to $D_\alpha u=0$ in the two parameter ranges $\alpha>-1$ and $\alpha\le -1$.
The following result is due to Huber~\cite{HuberUnique},
stated here using our choice of notation but otherwise unchanged.

\begin{thm}[A. Huber]\label{thm:Huberuniqueness}
Let $u$ be a solution of $D_\alpha u=0$, defined in a region $G$, the boundary
of which contains an open subset $S$ of $\{(x,y)\in\R^{n+1}: y=0\}$. If $u$ assumes the
boundary value $0$ on $S$, then we may conclude
\begin{enumerate}
\item[$(\mathrm{a})$] for $\alpha\le -1: u\equiv 0$ throughout $G$,
\item[$(\mathrm{b})$] for $\alpha>-1:u$ can be represented in the form $u=y^{\alpha+1}v(x,y)$,
where $v$ is analytic on $G\cup S$ and satisfies $D_{-(2+\alpha)}v=0$. Conversely
each function of this type fulfills the above hypotheses.
\end{enumerate}
\end{thm}

In particular, this result implies that a Green's function for $\Hu$ does not exist when $\alpha\le -1$,
while for $\alpha>-1$ Green's function is known, see Weinstein~\cite{Weinstein}
and Diaz and Weinstein~\cite{DiazWeinstein} for the case $n=1$ and
$n\ge2$, respectively.
We shall therefore henceforth restrict our attention to the parameter range $\alpha>-1$.
Note also that Theorem \ref{thm:Huberuniqueness} can be viewed as a uniqueness result
since it indicates how far a solution is determined by its boundary values. We
shall return briefly to the question of uniqueness for the Dirichlet problem \eqref{Dirichletproblem}
at the end of Section \ref{sec:asymptoticgrowth}.

In what follows, we will permit us to write $x^2$ to denote $x_1^2+\cdots +x_n^2$ whenever $x\in\R^n$.
We will also assume that all function spaces under consideration below are
defined on $\R^n$ unless explicitly stated otherwise.

\begin{dfn}\label{def:Poissonkernel}
Let $\alpha>-1$. Define the kernel $\K_\alpha$ by 
\begin{equation*}
\K_\alpha(x,y)= \frac{\Gamma((\alpha+n+1)/2)}{\Gamma((\alpha+1)/2)\pi^{n/2}}\cdot
\frac{y^{\alpha+1}}{(x^2+y^2)^{(\alpha+n+1)/2}},
\quad (x,y)\in\Hu,
\end{equation*}
where $\Gamma(s)=\int_0^\infty t^{s-1}e^{-t}dt$ for $s>0$ is the Gamma function.
\end{dfn}

Note that the classical Poisson kernel for the half space $\Hu$,
\[
P(x,y)=\K_0(x,y)=\frac{\Gamma((n+1)/2)}{\pi^{(n+1)/2}}\cdot
\frac{y}{(x^2+y^2)^{(n+1)/2}},
\quad (x,y)\in\Hu,
\]
is obtained
for the parameter value $\alpha=0$. 
Note also that when $\alpha=n-1$ we recover the Poisson kernel for the half space model $\mathbb H^{n+1}=\mathbb H^{\alpha+2}$
of real hyperbolic space
\[
P_{\mathbb H^{n+1}}(x,y)=\K_{n-1}(x,y)=\frac{\Gamma(n)}{\Gamma(n/2)\pi^{n/2}}\cdot
\frac{y^{n}}{(x^2+y^2)^{n}},
\quad (x,y)\in\Hu,
\]
see Guivarc'h, Ji, and Taylor~\cite{GuJiTa}.
Similarly, the family of kernels $\K_\alpha$ are naturally related to the differential operators $D_\alpha$
for $\alpha>-1$, see Theorem \ref{Poissonkernel} below. In preparation for the proof,
we calculate the $L^1$ norm of the function $\K_{\alpha,y}$.

\begin{lemma}\label{L1means}
Let $\alpha>-1$. Let $\K_{\alpha}$ be given by
Definition \ref{def:Poissonkernel}. Then
\[
\lVert \K_{\alpha,y}\rVert_{L^1}=\int \K_{\alpha,y}(x)dx=1,
\]
where $\K_{\alpha,y}$ is defined in accordance with \eqref{uyfcn}.
\end{lemma}

\begin{proof}
Define the auxiliary function $u$ by
\[
u(x,y)=\frac{y^{\alpha+1}}{(x^2+y^2)
^{(\alpha+n+1)/2}},
\quad (x,y)\in\Hu.
\]
In view of Definition \ref{def:Poissonkernel}, the theorem follows 
if we show that $u$ satisfies
\[
\lVert u_y\rVert_{L^1}= \frac{\Gamma((\alpha+1)/2)\pi^{n/2}}{\Gamma((\alpha+n+1)/2)},
\quad y>0,
\]
where $u_y$ is defined in accordance with \eqref{uyfcn}. To this end,
we first note that for each $y>0$ we have $u_y(x)>0$ for all $x\in\R^n$, and
that a change of variables $x/y\mapsto x$ shows that
$\int u_y(x)dx=\int u_1(x)dx$. 
Switching to spherical coordinates, we have
\[
\| u_1\|_{L^1}=\int \frac{dx}{(1+x^2)^{(\alpha+n+1)/2}}=\omega_{n-1}\int_{0}^{\infty}\frac{r^{n-1}dr}{(1+r^2)^{(\alpha+n+1)/2}},
\]
where $\omega_{n-1}=2\pi^{n/2}/\Gamma(n/2)$ is the surface area of the unit sphere $S^{n-1}$.
The change of variables $r=\sqrt{1/t-1}$ and a straightforward computation shows that 
\begin{align*}
2\int_0^\infty\frac{r^{n-1}dr}{(1+r^2)^{(\alpha+n+1)/2}}
&=\int_0^1 t^{(\alpha+1)/2-1}
(1-t)^{n/2-1}\\
&=B((\alpha+1)/2,n/2)
 =\frac{\Gamma((\alpha+1)/2)\Gamma(n/2)}{\Gamma((\alpha+n+1)/2)},
\end{align*}
where $B(x,y)$ is the Beta integral, see Andrews, Askey, and Roy~\cite[Theorem $1.1.4$]{Askey}.
Hence,
\[
\| u_1\|_{L^1}=\frac{\pi^{n/2}}{\Gamma(n/2)}\cdot 2\int_{0}^{\infty}\frac{r^{n-1}dr}{(r^2+1)^{(\alpha+n+1)/2}}
=\frac{\pi^{n/2}\Gamma((\alpha+1)/2)}{\Gamma((\alpha+n+1)/2)},
\]
which completes the proof.
\end{proof}

\begin{thm}\label{Poissonkernel}
Let $\alpha>-1$. Then the function $\K_\alpha$ given by Definition \ref{def:Poissonkernel}
is a solution to the equation $D_\alpha u=0$ in $\Hu$, 
and has the boundary limit $\lim_{y\to0}\K_{\alpha,y}=\delta_0$ 
in $\Es'$, where $\K_{\alpha,y}$ is defined in 
accordance with \eqref{uyfcn}.
\end{thm}

\begin{proof}
That $D_\alpha \K_\alpha(x,y)=0$ for $(x,y)\in\Hu$
follows by straightforward differentiation.
We proceed to analyze the boundary limit of $\K_\alpha$. For $y>0$ we have
\begin{equation}\label{eq:formulaPoissonkernel1}
\K_{\alpha,y}(x)=\K_\alpha(x,y)=
y^{-n}\K_{\alpha,1}(x/y),
\quad x\in\R^n,
\end{equation}
a fact that was already used in the proof of Lemma \ref{L1means}. By 
the same lemma, the function $\K_{\alpha,1}$ satisfies $\int \K_{\alpha,1}(x)dx=1$.
A standard construction of approximate  
identities now ensures that $\lim_{y\to0}\K_{\alpha,y}=\delta_0$  
in $\Es'$,
see Katznelson~\cite[Section~VI.1.13]{Katznelson} 
or H{\"o}rmander~\cite[Theorem~1.3.2]{Hormander1}.
\end{proof}

By virtue of Definition \ref{def:Poissonkernel} and Theorem \ref{Poissonkernel},
it is clear that the composition
$$
v:(x,y)\mapsto \K_\alpha(rx+t,ry),\quad r>0,\ t\in\R^n,
$$
also solves the equation $D_\alpha v=0$
in $\Hu$. This structural property is in fact shared by all solutions to $D_\alpha u=0$.

\begin{prop}
Let $\alpha>-1$, and let $u$ be a solution to $D_\alpha u=0$ in $\Hu$. Let $r>0$ and $t\in\R^n$ and set $v(x,y)=u(rx+t,ry)$.
Then $D_\alpha v=0$ in $\Hu$.
\end{prop}

\begin{proof}
If $(x,y)\in\Hu$ then $(rx+t,ry)\in\Hu$ when $r>0$ and $t\in\R^n$. Differentiation gives
\begin{align*}
D_\alpha v(x,y)&=r^2y^{-\alpha}\Delta u(rx+t,ry)-r\alpha y^{-\alpha-1}\partial u(rx+t,ry)/\partial y
\\& =r^{\alpha+2}((ry)^{-\alpha}\Delta u(rx+t,ry)-\alpha(ry)^{-\alpha-1}\partial u(rx+t,ry)/\partial y)
\\& =r^{\alpha+2}D_\alpha u(rx+t,ry)=0,
\end{align*}
which completes the proof.
\end{proof}

Next we analyze the Fourier transform of $x\mapsto \K_{\alpha,y}(x)$ for $\alpha>-1$, that is, the function
\begin{equation}\label{eq:Ftransformdef}
\xi\mapsto \widehat{\K_{\alpha,y}}(\xi)
=\int e^{-i\langle x,\xi\rangle} \K_{\alpha}(x,y) dx.
\end{equation}
A formula for this function has been obtained before, and can for example be found in
the paper by Baldi, Casadio Tarabusi and Fig{\`a}-Talamanca 
\cite[Section 4]{BaldiEtAl2} for $n=1$
within the framework of hyperbolic Brownian motion with drift,
although this requires some translation between the choice of notation.
(In this context, the terminology {\it characteristic function} is commonly used.)
We therefore prefer to include a direct proof using a different method, and we will indicate
the connection afterwards. A similar approach was used by Byczkowski, Graczyk and St{\'o}s
\cite{ByczkowskiEtAl} in the special case $\alpha=n-1$ corresponding to
classical hyperbolic Brownian motion on $\mathbb H^{n+1}$ (compare with
the discussion in the introduction above).

Recall that $\K_{\alpha,y}$ has bounded $L^1$ norm (independent of $y>0$)
by Lemma \ref{L1means}, so the integral \eqref{eq:Ftransformdef} is absolutely convergent and
$\widehat{\K_{\alpha,y}}(\xi)$
is a continuous function of $\xi$. 
To analyze how 
$\widehat{\K_{\alpha,y}}$
depends on $y$,
note that \eqref{eq:formulaPoissonkernel1} yields the identity
\begin{equation}\label{Fouriertransformidentity}
\widehat{\K_{\alpha,y}}(\xi)
=\int e^{-i\langle x,\xi\rangle}\K_{\alpha,1}(x/y) y^{-n}dx 
=\int  e^{-i\langle x,y\xi\rangle}\K_{\alpha,1}(x)dx
=\widehat{\K_{\alpha,1}}(y\xi).
\end{equation}
Since $x\mapsto \K_{\alpha,1}(x)$ is a radial function, 
the Fourier transform $\eta\mapsto 
\widehat{\K_{\alpha,1}}(\eta)$
is also a radial function.
In fact, let $f_\alpha:[0,\infty)\to\C$ be the $\Ci^\infty$ function defined by
\begin{equation}\label{eq:falfa}
f_\alpha(r)=\frac{\Gamma((\alpha+n+1)/2)}{\Gamma((\alpha+1)/2)\pi^{n/2}}\cdot
\frac{1}{(1+r^2)^{(\alpha+n+1)/2}},
\quad r\ge0,
\end{equation}
so that $\K_{\alpha,1}(x)=f_\alpha(\lvert x\rvert)$, and let $J_\nu$ denote the Bessel function {\it of the first kind} of order $\nu$,
\begin{equation}\label{besselcoefficient.prelim}
J_\nu(z)=\frac{(z/2)^\nu}{\Gamma(\nu+\frac{1}{2})\Gamma(\frac{1}{2})}\int_{-1}^1 e^{izt}(1-t^2)^{\nu-\frac{1}{2}}dt,
\quad \re \nu>-1/2,
\end{equation}
see for example equation (3) on p.~25 in the treatise by Watson~\cite{Watson}. Then we have 
$\widehat{\K_{\alpha,1}}(\eta)
=F_\alpha(|\eta|)$ where
\begin{equation}\label{eq:steinandweiss}
F_\alpha(r)=(2\pi)^{n/2}r^{(2-n)/2}\int_0^\infty f_\alpha(s)s^{n/2} J_{(n-2)/2}(rs)ds,
\end{equation}
see Stein and Weiss~\cite[Chapter 4, Theorem 3.3]{SteinWeiss}. Note that their definition
of the Fourier transform differs from ours by a scaling factor, which explains
the difference in appearance between the formulas.
Thus, for $\alpha>-1$ the Fourier transform of $x\mapsto \K_{\alpha,y}(x)$ can be written as
\begin{equation}\label{aux:Fouriertransform}
\widehat{\K_{\alpha,y}}(\xi)=
\frac{2^{n/2}\Gamma((\alpha+n+1)/2)}{\Gamma((\alpha+1)/2)}
(y\lvert\xi\rvert)^{(2-n)/2}\int_0^\infty \frac{s^{n/2}J_{(n-2)/2}(sy|\xi|)}{(1+s^2)^{
(\alpha+n+1)/2}}ds.
\end{equation}

By using the properties of $\K_\alpha$ given by Theorem \ref{Poissonkernel},
we may (indirectly) evaluate the integral in \eqref{aux:Fouriertransform}.
We thus return to the Fourier transform of $x\mapsto \K_{\alpha,y}(x)$ given by
\eqref{eq:Ftransformdef}. In view of \eqref{Fouriertransformidentity}, it is sufficient
to study the case $y=1$.
Let therefore $F_\alpha$ be the radial function defined above such that
$\widehat{\K_{\alpha,1}}(y\xi)
=F_\alpha(y|\xi|)$.
Since $\widehat{\K_{\alpha,1}}$
is continuous and $\widehat{\K_{\alpha,1}}(\eta)\to0$ as $|\eta|\to\infty$ by the Riemann Lebesgue lemma, we may
identify the map $(\xi,y)\mapsto \widehat{\K_{\alpha,y}}(\xi)=\widehat{\K_{\alpha,1}}(y\xi)=F_\alpha(y|\xi|)$
with a distribution in $\Es'(\Hu)$.
Now $\K_\alpha$ satisfies $D_\alpha \K_\alpha=0$ in $\Hu$, so by a Fourier transformation with respect to the
$x$ variables we obtain
\begin{align}
0&=\F_x((\nabla_x \cdot (y^{-\alpha}\nabla_x)+\partial_y y^{-\alpha}\partial_y)\K_{\alpha,y})(\xi)\\
&=y^{-\alpha}\bigg(\frac{\partial^2 }{\partial y^2}-\frac{\alpha}{y}\frac{\partial }{\partial y}-\xi^2\bigg) F_\alpha(y\lvert\xi\rvert),
\label{Fourierdiffeq}
\end{align}
which we interpret in the distributional sense. However, by H{\"o}rmander~\cite[Theorem 4.4.8]{Hormander1}
any distribution in $\De'(\Hu)$ solving the differential equation in the right-hand side of \eqref{Fourierdiffeq}
is a $\Ci^2$ function of $y$ with values in $\De'$; in our case, the values will even be in $\Es'$.
Indeed, for fixed $y>0$ we can identify $\K_{\alpha,y}$ with a tempered distribution in $\Es'$,
so its Fourier transform also belongs to $\Es'$ which proves the claim.
If we perform the differentiations in \eqref{Fourierdiffeq} we find that
\begin{align*}
0&=y^{-\alpha}\bigg(\frac{\partial^2 }{\partial y^2}-\frac{\alpha}{y}\frac{\partial }{\partial y}-\xi^2\bigg) F_\alpha(y|\xi|)\\
&=y^{-\alpha}\xi^2\bigg(F_\alpha''(y|\xi|)-\frac{\alpha}{y|\xi|}F_\alpha'(y|\xi|)-F_\alpha(y|\xi|)\bigg),
\end{align*}
so we are led to consider the ordinary differential equation
\begin{equation}\label{BesselequationOur}
v''(t)-\frac{\alpha}{t}v'(t)-v(t)=0.
\end{equation}
We therefore digress and recall some well-known facts concerning this equation. For a more thorough discussion
we refer the reader to the work by Watson~\cite{Watson} and the references therein.

If $v$ is a solution to \eqref{BesselequationOur}, set $u(t)=v(t)t^{-\nu}$ with $\nu=\frac{\alpha}{2}+\frac{1}{2}$.
It is straightforward to check that $u$ then solves the equation
\begin{equation}\label{purelyimaginary}
t^2u''(t)+tu'(t)-(t^2+\nu^2)u(t)=0.
\end{equation}
This implies that the general solution
to \eqref{BesselequationOur} is given by $t\mapsto v(t)=t^\nu u(t)$, where $\nu=\frac{\alpha}{2}+\frac{1}{2}$
and $u$ is a general solution to \eqref{purelyimaginary}.
Recall that the pair $I_\nu$ and $K_\nu$ of modified Bessel functions {\it of the third kind}
always form a fundamental system of solutions to \eqref{purelyimaginary},
see Watson~\cite[\S \! 3$\cdot$7]{Watson}.
Here 
\begin{equation}\label{besselcoefficient}
I_\nu(z)=\sum_{m=0}^\infty\frac{(z/2)^{\nu+2m}}{m!\Gamma(\nu+m+1)},\quad\nu\in\C,
\end{equation}
and
\begin{equation}\label{besselcoefficient2}
K_\nu(z)=\frac{\pi}{2}\cdot\frac{I_{-\nu}(z)-I_\nu(z)}{\sin\nu\pi},\quad\nu\in\C.
\end{equation}
Moreover, if $\nu>0$ then $I_\nu(z)$ tends to infinity while
$K_\nu(z)$ tends exponentially to zero as $z\to\infty$ through positive values, see Watson~\cite[\S \! 7$\cdot$23]{Watson}.

\begin{rmk}
Equation \eqref{BesselequationOur} can be derived from Bessel's equation for functions of order $\nu=\frac{\alpha}{2}+\frac{1}{2}$,
\begin{equation}\label{BesselequationStandard}
z^2\frac{d^2 u}{dz^2}+z\frac{du}{dz}+(z^2-\nu^2)u=0,
\end{equation}
by elementary transformations of the
dependent and independent variables. 
Indeed, as we have seen we can transform \eqref{BesselequationOur} to
\eqref{purelyimaginary},
which differs from Bessel's equation only in the coefficient of $u$.
By the change of variables $t\mapsto it$, \eqref{purelyimaginary} is transformed to Bessel's equation
for functions of order $\nu$, so its general solution is given by $z\mapsto u(iz)$ where $u$
is a general solution to \eqref{BesselequationStandard}. This implies that the general solution
to \eqref{BesselequationOur} can also be obtained as $t\mapsto v(t)=t^\nu u(it)$, where $\nu=\frac{\alpha}{2}+\frac{1}{2}$
and $u$ is a general solution to \eqref{BesselequationStandard}.
Since the pair $H_\nu^{(1)}$ and $H_\nu^{(2)}$ of Bessel functions {\it of the third kind}
always form a fundamental system of solutions to Bessel's equation for functions of order $\nu$,
see Watson~\cite[\S \! 3$\cdot$63]{Watson}, 
it is thus possible to represent the Fourier transform of $\K_{\alpha,y}$
also in terms of this pair of solutions (due to considerations of growth,
it turns out to be expressed by means of $H_\nu^{(1)}$, $\nu=\frac{\alpha}{2}+\frac{1}{2}$).
We will not pursue this direction further.
\end{rmk}

\begin{thm}\label{thm:Ftransform}
Let $\alpha>-1$. Then the Fourier transform of $\K_{\alpha,y}$ given by \eqref{eq:Ftransformdef}
can be expressed in terms of Bessel functions,
\begin{equation}\label{eq:Ftransform}
\widehat{\K_{\alpha,y}}(\xi)
=\frac{2^{(1-\alpha)/2}}{\Gamma({(\alpha+1)/2})}
(y|\xi|)^{(\alpha+1)/2}
K_{(\alpha+1)/2}(y|\xi|),
\quad y>0,
\end{equation}
where $K_{(\alpha+1)/2}$ is the modified Bessel function of the third kind of order $(\alpha+1)/2$.
\end{thm}

\begin{proof}
To shorten notation, let $\nu=\frac{\alpha}{2}+\frac{1}{2}$ and note that $\nu>0$ by assumption.
Let $F_\alpha$ be the function satisfying $F_\alpha(y|\xi|)=
\widehat{\K_{\alpha,1}}(y\xi)$.
By the discussion preceding the theorem, it follows that
since $F_{\alpha}$ solves \eqref{BesselequationOur}, we have
\begin{equation}\label{AandBconstants}
F_\alpha(r)=A_\alpha r^\nu I_\nu(r)
+ B_\alpha r^\nu K_\nu(r),\quad r\ge0,
\end{equation}
for some constants $A_\alpha$ and $B_\alpha$ which are to be determined.
Since $\widehat{\K_{\alpha,1}}$
is bounded by virtue of Lemma \ref{L1means}, and $K_\nu(r)$ tends exponentially to zero while $I_\nu(r)\to\infty$  
as $r\to\infty$, the coefficient of
$I_\nu$ in \eqref{AandBconstants} must vanish, that is, $A_\alpha=0$.

To determine $B_\alpha$, suppose first that $\alpha>-1$ is not an odd positive integer,
so that $\nu\notin\Z_+$. 
By an application of Theorem \ref{Poissonkernel} we have that
$\widehat{\K_{\alpha,y}}\to1$ in $\Es'$ as $y\to0$. Since
$\widehat{\K_{\alpha,y}}(\xi)$
is a continuous function of $\xi$
we find that $\widehat{\K_{\alpha,y}}(\xi)\to1$
as $y\to0$ for all $\xi\in\R$, which by virtue of \eqref{Fouriertransformidentity} implies
that $\widehat{\K_{\alpha,1}}(y)=\widehat{\K_{\alpha,y}}(1)\to1$ as $y\to0$.
In particular, this means that $F_\alpha(r)\to 1$ as $r\to0$.
In view of \eqref{besselcoefficient} and \eqref{besselcoefficient2} this gives
\[
1=\frac{B_\alpha 2^{\nu-1}\pi}{\sin\nu\pi\, \Gamma(1-\nu)}
\]
since $\nu>0$.
Invoking Euler's reflection formula $\Gamma(z)\Gamma(1-z)\sin \pi z=\pi$
we find that $B_\alpha=2^{1-\nu}/\Gamma(\nu)$.
Since $\nu=\frac{\alpha}{2}+\frac{1}{2}$,
this completes the proof in the case when $\alpha$ is not an odd positive integer.
In view of Definition \ref{def:Poissonkernel},
an application of the dominated convergence theorem shows that
$\widehat{\K_{\alpha,1}}(\xi)\to \widehat{\K_{2k-1,1}}(\xi)$ as $\alpha\to 2k-1$.
By Watson~\cite[\S \! 3$\cdot$7]{Watson}, $K_\nu(z)\to K_{k}(z)$ as $\nu=\frac{\alpha}{2}+\frac{1}{2}\to k$,
so the general case follows by continuity. This completes the proof.
\end{proof}

Note that since $K_\nu(z)$ is an analytic function of $z$,
Theorem \ref{thm:Ftransform} shows that the Fourier transform of $\K_{\alpha,y}$
fails to be smooth at the origin. That $\widehat{\K_{\alpha,y}}$ cannot be smooth
on all of $\R^n$ can of course also been seen directly from Definition \ref{def:Poissonkernel}
in view of how the Fourier transform relates integrability and regularity.
We also remark that $K_{1/2}(z)=\sqrt{\frac{\pi}{2z}}e^{-z}$,
so when $\alpha\to0$ we recover the Fourier transform
of the Poisson kernel for the upper half space.
Furthermore, a comparison with formulas (4.5) and (4.6) in
Baldi et al.~\cite{BaldiEtAl2}
(using $\nu=\frac{\alpha}{2}+\frac{1}{2}$) shows that we recover their result
concerning the Fourier transform of the Poisson kernel
of the infinitesimal generator associated to \eqref{HyperbolicBrownianMotion} in the case $n=1$.
The special case $\alpha=n-1$ appears in
Byczkowski et al.~\cite{ByczkowskiEtAl}, see the proof of their Theorem 2.1.

Next, we give an integral representation of
$\widehat{\K_{\alpha,y}}$. There is of course a wide variety of forms of the Bessel function $K_\nu$
which can be used to express \eqref{eq:Ftransform},
but we will content ourselves with the following result which
proves to be useful later.

\begin{cor}\label{cor:Ftransform}
Let $\alpha>-1$. Then the Fourier transform of $\K_{\alpha,y}$ 
given by \eqref{eq:Ftransformdef}
can be expressed as
\begin{equation}\label{eq:Ftransformasintegral}
\widehat{\K_{\alpha,y}}(\xi)
=\frac{(y\lvert\xi\rvert)^{\alpha+1}}{\Gamma(\alpha+1)}
\int_1^\infty e^{-y\lvert\xi\rvert t} (t^2-1)^{\alpha/2}dt,
\quad y>0.
\end{equation}
\end{cor}

\begin{proof}
By Watson~\cite[\S \! 6$\cdot$15]{Watson}, identity $(4)$, we have the representation
\[
K_\nu(z)=\frac{\Gamma(\frac{1}{2})(\frac{1}{2}z)^\nu}{\Gamma(\nu+\frac{1}{2})}\int_1^\infty e^{-zt}(t^2-1)^{\nu-\frac{1}{2}}dt.
\]
In view of the duplication formula $\Gamma(\nu)\Gamma(\nu+\frac{1}{2})=
2^{1-2\nu}\Gamma(\frac{1}{2})\Gamma(2\nu)$, the result is now an immediate consequence of Theorem \ref{thm:Ftransform}.
\end{proof}

\section{Weighted spaces of distributions}\label{section:Weighteddistributions}

In this section we recall certain facts concerning weighted spaces of distributions,
and recall the definition of the $\Es'$-convolution.
We also prove some auxiliary results that will be used in the next section when
we solve the Dirichlet problem \eqref{Dirichletproblem}.
We mention that the weighted spaces of distributions that we will consider appear naturally
in the context of Newtonian potentials of distributions, see Schwartz~\cite{Schwartz}, and have subsequently been studied
by many authors. They were recently used
in Alvarez et al.~\cite{Alvarezetal1} to characterize the tempered distributions that are $\Es'$-convolvable with the classical Poisson kernel
for the half space, and
in Alvarez et al.~\cite{Alvarezetal2}
to study harmonic extensions of distributions.
For further details on how these spaces appear, as well as on the
$\Es'$-convolution and other notions of convolution of tempered distributions,
we refer to the mentioned papers and the references therein.

We begin by recalling the definitions and some properties of spaces of distributions considered
by Laurent Schwartz. For details we refer to Schwartz~\cite[Chapitre VI, \S 8]{Schwartz}.
To make the notation less cumbersome we will as before assume that all the spaces under consideration below are
defined on $\R^n$ unless explicitly stated otherwise. We let $\De_{L^p}$ denote the
vector space of smooth functions $\varphi\in\Ci^\infty$ such that all derivatives $\partial^\beta\varphi$
belong to $L^p$. We endow $\De_{L^p}$ with the topology in which a sequence $\{\varphi_j\}_{j=1}^\infty$
converges to 0 in $\De_{L^p}$ if $\partial^\beta\varphi_j\to 0$ in $L^p$ for all multi-indices $\beta$. 
$\De_{L^p}$ is then a locally convex, complete topological vector space. We employ the notation
$\Bo$ for the special case $p=\infty$, that is, $\Bo=\De_{L^\infty}$ is the space of
smooth functions $\varphi:\R^n\to\C$ such that $\varphi$ and all derivatives $\partial^\beta\varphi$
are bounded. We will let $\Bop$ denote the closed subspace consisting of those elements $\varphi\in\Bo$ such that
$\partial^\beta\varphi(x)\to0$ as $|x|\to\infty$ for all multi-indices $\beta$.
We have the continuous strict inclusions
\[
\De_{L^p}\subset \De_{L^q}\subset \Bop,\quad 1\le p< q<\infty.
\]
Moreover, the space $\Ci_0^\infty$ of compactly supported smooth functions is dense in $\De_{L^p}$ for $1\le p<\infty$, and in $\Bop$,
but $\Ci_0^\infty$ is not dense in $\Bo$.
For this reason we will also endow $\Bo$ with the finest locally convex topology that on bounded subsets of $\Bo$
induces the topology inherited from $\Ci^\infty$, and this space
will be denoted by $\Boc$. We have that $\Ci_0^\infty$ is dense in $\Boc$, and
since $\Ci_0^\infty$ is also dense in $\Bop$, it follows that $\Bop$ is dense in $\Boc$.

For $1<p\le\infty$ we let $\De_{L^p}'$ denote the dual of $\De_{L^{p'}}$ where $p'$ is
the conjugate exponent of $p$, and we let $\De_{L^1}'$ denote the dual of
$\Bop$. Due to the dense inclusions $\Ci_0^\infty\subset\De_{L^{p'}}\subset\Bop$
for $1\le p'<\infty$, it follows that $\De_{L^p}'$ is a space of distributions for $1\le p\le\infty$.
(This is not true for the dual of $\Bo$.) We have the continuous strict inclusions
$\De_{L^p}\subset L^p\subset\De_{L^p}'$, as well as
$\De_{L^p}'\subset \De_{L^q}'$ for $p< q$.

By Schwartz~\cite[Chapitre VI, Th{\'e}or{\`e}me XXV]{Schwartz}, a distribution $u$ belongs to $\De_{L^p}'$
if and only if the regularization $u\ast\varphi$ belongs to $\De_{L^p}$
for every $\varphi\in\Ci_0^\infty$, where the convolution is defined in the usual distributional sense.
Moreover, $u$ belongs to $\De_{L^p}'$ if and only if $u$ can be represented as a finite sum
\[
u=\sum_\beta\partial^\beta u_\beta,\quad u_\beta\in L^p,
\]
where the derivatives are interpreted in the distributional sense. Hence,
$\De_{L^p}'$ is continuously embedded in the space $\Es'$ of tempered distributions.
Consider now the case $p=1$, and let $u\in\De_{L^1}'$. Representing $u$ as a finite sum of
distributional derivatives of integrable functions allows for
$u$ to be extended to a continuous linear form on $\Boc$. Since $\Bop$ is dense in $\Boc$,
the extension is unique. Hence $\De_{L^1}'$ is the dual of $\Boc$, see Schwartz~\cite[p.~203]{Schwartz}. Moreover,
if $\langle\phantom{i},\phantom{i}\rangle_{\mathscr V',\mathscr V}$ in general denotes the duality pairing
between a topological space $\mathscr V$ and its dual $\mathscr V'$, the integral of $u\in\De_{L^1}'$ is well defined in the sense that
\[
\langle u,1\rangle_{\De_{L^1}',\Boc}=\sum_\beta (-1)^{|\beta|}\int u_\beta(x)\partial^\beta 1 dx=\int u_0(x)dx.
\]
If $u\in\De_{L^1}'$ also belongs to $L^1$, then $\langle u,1\rangle_{\De_{L^1}',\Boc}$ coincides with the integral of $u$.
The distributions in $\De_{L^1}'$ are therefore sometimes called integrable distributions.

\begin{dfn}\label{def:weightedspace}
Let $w:\R^n\to\R$ be defined by $w(x)=(1+x^2)^{\frac{1}{2}}$, and let $\mu$ and $p$ be real parameters
with $1\le p<\infty$. We define the weighted space of distributions $w^\mu\De_{L^p}'$ as
\[
w^\mu\De_{L^p}'=\{u\in\Es':w^{-\mu}u\in\De_{L^p}'\}
\]
with the topology induced by the map from $w^{\mu}\De_{L^p}'$ to $\De_{L^p}'$ given by $u\mapsto w^{-\mu}u$.
\end{dfn}

Note that $(x,\xi)\mapsto w^\mu(\xi)$ belongs the Kohn-Nirenberg symbol class
of order $\mu$, that is, for any multi-indices $\beta$ and $\gamma$
we can find a constant $C_{\beta,\gamma}$ such that
\begin{equation}\label{symbol}
|\partial_x^\beta\partial_\xi^\gamma w^\mu(\xi)|\le C_{\beta,\gamma}(1+|\xi|)^{\mu-|\gamma|},\quad\xi\in\R^n,
\end{equation}
see H{\"o}rmander~\cite[Definition 18.1.1]{Hormander3}. In particular, $w^\mu$
is a so-called order function for any $\mu\in\R$, that is, $\partial^\beta w^\mu=\mathcal O(w^\mu)$ for any multi-index $\beta$.
Since differentiation is a linear continuous operation
on $\De_{L^p}'$, see Schwartz~\cite[p.~200]{Schwartz}, 
it follows that the space $w^\mu\De_{L^p}'$ is also closed under differentiation.

For $1<p<\infty$, the space $w^\mu\De_{L^p}'$ is the dual of $w^{-\mu}\De_{L^{p'}}$ where
$p'$ is the conjugate exponent of $p$. The space $w^\mu\De_{L^1}'$ is the dual of $w^{-\mu}\Bop$ and $w^{-\mu}\Boc$.
Here the weighted spaces $w^{-\mu}\De_{L^{p'}}$ are defined in the analog way as the functions $\varphi\in\Ci^\infty$
such that $w^\mu\varphi\in \De_{L^{p'}}$.
The duality is naturally given by
\[
\langle u, \varphi\rangle_{w^\mu\De_{L^p}',w^{-\mu}\De_{L^{p'}}}
=\langle w^{-\mu}u, w^\mu\varphi\rangle_{\De_{L^p}',\De_{L^{p'}}},
\]
and similarly for  $w^{-\mu}\Bop$ and $w^{-\mu}\Boc$.
However, since $w^\mu$ is an order function it is easy to see that
$\varphi\in w^{\mu}\De_{L^{p}}$ if and only if 
$w^{-\mu}\partial^\beta\varphi\in L^{p}$ for all multi-indices $\beta$, that is,
\[
w^{\mu}\De_{L^{p}}= \De_{L^{p}(w^{-\mu p})},
\]
where $L^{p}(w^{-\mu p})=L^{p}(w^{-\mu p}(x)dx)$ is the weighted $L^{p}$ space defined in the natural way.
It follows that $L^p(w^{-\mu p})$ is continuously embedded in $w^\mu\De_{L^p}'$ for $\mu\in\R$ and $1\le p<\infty$. In particular
\begin{equation}\label{L1inclusions}
w^\mu\De_{L^1}\subset L^1(w^{-\mu })\subset w^\mu\De_{L^1}',\quad\mu\in\R.
\end{equation}
Conversely, we recall the following useful representation formula for distributions in $w^\mu\De_{L^p}'$, see Alvarez et al.~\cite[Lemma 3.3]{Alvarezetal2}
and~\cite[Remark 3.4]{Alvarezetal2}. For $\mu\in\R$ and $1\le p<\infty$ we have that
\begin{equation}\label{representationformula}
w^\mu\De_{L^p}'=\bigg\{u\in\Es':u=\sum_\beta \partial^\beta u_\beta,\quad u_\beta\in L^p(w^{-\mu p})\bigg\},
\end{equation}
where the summation is over a finite set.
Pointwise multiplication is well defined and continuous from $\Bo\times\Bo$ to $\Bo$ and from
$\Bop\times\Bo$ to $\Bop$, so $\De_{L^1}'$ is closed under multiplication by functions in $\Bo$.
Using the representation formula \eqref{representationformula} for $p=1$, it follows that we have the continuous strict inclusions
\[
w^{\mu_1}\De_{L^1}'\subset w^{\mu_2}\De_{L^1}'\subset\Es',\quad\mu_1<\mu_2.
\]

We next recall the definition of the so-called $\Es'$-convolution
proposed by Hirata and Ogata~\cite{HirataOgata}. The definition was given an equivalent form by Shiraishi~\cite{Shiraishi}, 
which is the one we will use.

\begin{dfn}\label{def:convolution}
Two tempered distributions $u$ and $v$ in $\Es'$ are said to be $\Es'$-convolvable
if the multiplicative product $u(\check v\ast\varphi)$ belongs to $\De_{L^1}'$ for every $\varphi\in\Es$.
Then the map from $\Es$ to $\C$ given by
\[
\varphi\mapsto\langle u(\check v\ast\varphi),1\rangle_{\De_{L^1}',\Boc}
\]
is linear and continuous, and thus defines a tempered distribution denoted by $u\ast v$.
\end{dfn}
Here, $\check\varphi(x)=\varphi(-x)$ for $\varphi\in\Es$ and we extend this operation to $\Es'$
by duality. If $\tau_\eta\varphi(x)=\varphi(x-\eta)$ denotes translation, then  $\check v\ast\varphi$ is the function
\[
\eta\mapsto\langle\check v,\tau_\eta\check\varphi\rangle_{\Es',\Es}
=\langle v,\tau_{-\eta}\varphi\rangle_{\Es',\Es}.
\]
We remark that when defined, the $\Es'$-convolution of $u$ and $v$
is commutative, and satisfies the Fourier exchange formula $(u\ast v)\hat {\ }=\hat u\hat v$.
The notation $u\ast v$ for the $\Es'$-convolution of $u$ and $v$
is justified by the fact that Definition \ref{def:convolution} coincides with the usual
definition of convolution in the sense of distributions
whenever the latter definition is applicable.

We now turn to the problem of finding the optimal class of tempered distributions that are
$\Es'$-convolvable with the kernel $\K_{\alpha,y}$. To simplify the notation below, we introduce the function $\wa$ given by
\begin{equation}\label{w-alfa}
\wa(x)=w^{\alpha+n+1}(x)
=(1+x^2)^{(\alpha+n+1)/2},\quad x\in\R^n.
\end{equation}
We will thus not signify the dependence on the dimension in the notation.
We remark that $\wa^{-1}$ is modulo a scaling factor equal to $\K_{\alpha,1}$.

Our goal in this section is to prove the following result, which contains Alvarez et al.~\cite[Theorem 10]{Alvarezetal1}
as the special case $\alpha=0$.

\begin{thm}\label{thm:convolutionwelldefined}
Let $\alpha>-1$. Let $\wa$ be given by \eqref{w-alfa}, and let $f\in \Es'$. Then the following
assertions are equivalent:
\begin{itemize}
\item[$(\mathrm{i})$] $f\in \wa\De_{L^1}'$.
\item[$(\mathrm{ii})$] $f$ is $\Es'$-convolvable with $\K_{\alpha,y}$ for each $y>0$.
\end{itemize}
Furthermore, if $f\in \wa\De_{L^1}'$ then $\K_{\alpha,y}\ast f$ is the function on $\R^n$
given by
\begin{equation}\label{convolutionformula}
\eta\mapsto\langle \wa^{-1}f,\wa \tau_\eta \K_{\alpha,y}\rangle_{\De_{L^1}',\Boc}
\end{equation}
for each $y>0$.
\end{thm}

\begin{proof}
We shall adapt a combination of the proofs of Alvarez et al.~\cite[Proposition 7]{Alvarezetal1} and~\cite[Theorem 10]{Alvarezetal1}
in the presence of a parameter $\alpha>-1$. 
Assume first that (i) holds, and let $f=\wa u$ for some $u\in\De_{L^1}'$.
To prove that the $\Es'$-convolution $\K_{\alpha,y}\ast f$ is well defined, we must according to Definition \ref{def:convolution} show that
the multiplicative product $f(\K_{\alpha,y}\ast\varphi)=u (\K_{\alpha,y}\ast\varphi)\wa$ belongs to $\De_{L^1}'$ for each $\varphi\in\Es$ and $y>0$.
Since pointwise multiplication is well defined and continuous from $\Bo\times\Bo$ to $\Bo$ (and from
$\Bop\times\Bo$ to $\Bop$), it follows that $\De_{L^1}'$ is closed under multiplication by functions in $\Bo$.
It therefore suffices to show that we have $(\K_{\alpha,y}\ast\varphi)\wa\in\Bo$ for each $\varphi\in\Es$ and $y>0$.

Recall Peetre's inequality
\begin{equation}\label{Peetre}
(1+|t-x|^2)^s\le2^{|s|}(1+x^2)^{|s|}(1+t^2)^s,\quad s\in\R.
\end{equation}
We have that $\partial^\beta ( \K_{\alpha,y}\ast\varphi)= \K_{\alpha,y}\ast\partial^\beta\varphi$, and by using \eqref{Peetre}
it is straightforward to check that
\[
|\K_{\alpha,y}\ast\partial^\beta\varphi(x)|\le\frac{C_{\alpha,n}}{y^n}\bigg(1+ \frac{x^2}{y^2}
\bigg)^{-(\alpha+n+1)/2}\int\bigg(1+ \frac{t^2}{y^2}
\bigg)^{(\alpha+n+1)/2}\partial^\beta\varphi(t) dt,
\]
where the constant $C_{\alpha,n}$ depends on $\alpha$ and $n$. Next, let $M_{\alpha,n}$ denote the multiplication operator
$M_{\alpha,n}\varphi(t)=|t|^{\alpha+n+1}\varphi(t)$. By splitting the integral in the right-hand side above into the two regions $|t|<y$ and
$|t|\ge y$, it is straightforward to check that this results in the estimate
\[
|\K_{\alpha,y}\ast\partial^\beta\varphi(x)|\le\frac{C_{\alpha,n}}{y^n}\bigg( 1+\frac{x^2}{y^2}
\bigg)^{-(\alpha+n+1)/2}( \|\partial^\beta\varphi\|_{L^1}+y^{-(\alpha+n+1)}\|M_{\alpha,n}\partial^\beta\varphi\|_{L^1})
\]
for some new constant $C_{\alpha,n}$ depending on $\alpha$ and $n$. Hence $(\K_{\alpha,y}\ast\varphi)\wa\in\Bo$,
so the $\Es'$-convolution $\K_{\alpha,y}\ast f$ exists for each $y>0$ when $f\in\wa\De_{L^1}'$. Moreover, we have
\[
\langle \K_{\alpha,y}\ast f,\varphi\rangle_{\Es',\Es}= \langle f( \K_{\alpha,y}\ast\varphi),1\rangle_{\De_{L^1}',\Boc}
=\langle \wa^{-1}f ,\wa(\K_{\alpha,y}\ast\varphi) \rangle_{\De_{L^1}',\Boc}.
\]
Using \eqref{representationformula} for $\wa^{-1}f =u\in\De_{L^1}'$, it is straightforward to check that the quantity
$\langle u ,\wa(\K_{\alpha,y}\ast\varphi) \rangle_{\De_{L^1}',\Boc}$ coincides with the integral
$\int \langle u ,\wa\tau_\eta \K_{\alpha,y}\rangle_{\De_{L^1}',\Boc}\varphi(\eta)d\eta$,
see the end of the proof of Alvarez et al.~\cite[Proposition 7]{Alvarezetal1} for details.
This proves that the $\Es'$-convolution $\K_{\alpha,y}\ast f$ is given by \eqref{convolutionformula}.

Assume next that (ii) holds, and fix $f\in\Es'$. By Alvarez et al.~\cite[Proposition 9]{Alvarezetal1}, $f\in\wa\De_{L^1}'$ if and only if
$f$ can be represented as a sum $f=f_1+M_{\alpha,n}f_2$, where $f_1\in\mathscr E'$, $M_{\alpha,n}$ is the
multiplication operator introduced above, and $f_2\in\De_{L^1}'$ is not supported near the origin.
Introduce a cutoff function $\chi\in\Ci_0^\infty$ taking values in $[0,1]$ such that $\chi(x)$ is identically equal to 1 for $|x|\le 1/2$,
positive for $|x|<1$ and vanishes for $|x|\ge1$. Write $f=\chi f+(1-\chi)f$, and note that $f_1=\chi f\in\mathscr E'$.

Next, set $\psi(t)=\chi(3t)$. Then $\psi(t)>0$ for $|t|<1/3$, and $\psi(t)=0$ for $|t|\ge1/3$. In particular,
$\psi$ vanishes on $\supp{(1-\chi)}=\{t\in\R^n:|t|\ge 1/2\}$. Consider the convolution
\[
\K_{\alpha,y}\ast\psi(x)=C_{\alpha,n}\int_{|t|<1/3} \frac{y^{\alpha+1}}{(|x-t|^2+y^2)^{(\alpha+n+1)/2}}
\psi(t)dt.
\]
For $|t|<1/3$ and $x\in\supp{(1-\chi)}$ we have $|x-t|\le 2|x|$, which implies that
\[
\K_{\alpha,y}\ast\psi(x)\ge C_{\alpha,n}\frac{y^{\alpha+1}}{(x^2+y^2)^{(\alpha+n+1)/2}}
\|\psi\|_{L^1},\quad x\in\supp{(1-\chi)}.
\]
It follows that $(1-\chi)\wa^{-1}(\K_{\alpha,y}\ast\psi)^{-1}\in\Bo$ for each $y>0$. Moreover,
assumption (ii) implies that we have $(\K_{\alpha,y}\ast\psi)f\in\De_{L^1}'$ by virtue of Definition \ref{def:convolution},
which gives
\begin{align*}
(1-\chi(x))f(x)&=|x|^{\alpha+n+1}
\frac{\wa(x)}{|x|^{\alpha+n+1}}\cdot
\frac{1-\chi(x)}{\wa(x)\K_{\alpha,y}\ast\psi(x)}\K_{\alpha,y}\ast\psi(x)f(x)\\
&=M_{\alpha,n}(x)f_2(x),
\end{align*}
where $f_2\in\De_{L^1}'$ since the map
$$
x\mapsto \frac{\wa(x)}{|x|^{\alpha+n+1}}=(1+|x|^{-2})^{(\alpha+n+1)/2},\quad x\in\supp{(1-\chi)},
$$
also belongs to $\Bo$,
and $\De_{L^1}'$ is closed under multiplication by functions in $\Bo$.
Since $f_2\equiv 0$ near the origin, this completes the proof.
\end{proof}

\begin{rmk}\label{rmk:smoothness}
We remark that the calculations in the proof of Theorem \ref{thm:convolutionwelldefined} show that if $f\in\wa\De_{L^1}'$
then the function $F:\Hu\to\C$ given by
\[
F(x,y)= \K_{\alpha,y}\ast f(x)= \langle \wa^{-1} f,\wa\tau_x \K_{\alpha,y}\rangle_{\De_{L^1}',\Boc},\quad(x,y)\in\Hu,
\]
is in $\Ci^\infty(\Hu)$ and satisfies $D_\alpha  F=0$ in $\Hu$. The derivatives
$\partial_x^\beta\partial_y^k F(x,y)$ are given by
\[
\partial_x^\beta\partial_y^k F(x,y)= \langle \wa^{-1}f,\wa\partial_x^\beta\partial_y^k \tau_x \K_{\alpha,y}\rangle_{\De_{L^1}',\Boc}.
\]
\end{rmk}

We end this section with the following proposition.

\begin{prop}\label{prop:diffofconvolution}
Let $\alpha>-1$. Let $f\in\wa\De_{L^1}'$. Then for any multi-index $\beta$ we have
$\partial^\beta(\K_{\alpha,y}\ast f)= \K_{\alpha,y}\ast (\partial^\beta f)$ in $\Es'$
for each $y>0$.
\end{prop}

\begin{proof}
Using the properties of the $\Es'$-convolution we have
\begin{align*}
\langle\partial^\beta(\K_{\alpha,y}\ast f),\varphi\rangle_{\Es',\Es}
&=(-1)^{|\beta|}\langle \K_{\alpha,y}\ast f,\partial^\beta\varphi\rangle_{\Es',\Es}\\
&=(-1)^{|\beta|}\langle f( \K_{\alpha,y}\ast \partial^\beta\varphi),1\rangle_{\De_{L^1}',\Boc}\\
&=(-1)^{|\beta|}\langle f ,\K_{\alpha,y}\ast \partial^\beta\varphi\rangle_{\De_{L^1}',\Boc},
\end{align*}
where the last expression makes sense, since $f$ can be written $f=\wa u$ for some $u\in\De_{L^1}'$,
and $( \K_{\alpha,y}\ast \partial^\beta\varphi )\wa$ belongs to $\Bo$ by the first part of the proof of
Theorem \ref{thm:convolutionwelldefined}. Since $\partial^\beta f\in\wa\De_{L^1}'$, we similarly have
\begin{align*}
\langle \K_{\alpha,y}\ast (\partial^\beta f),\varphi\rangle_{\Es',\Es}
&=\langle\partial^\beta f(\K_{\alpha,y}\ast\varphi),1\rangle_{\De_{L^1}',\Boc}\\
&=\langle\partial^\beta f, \K_{\alpha,y}\ast\varphi\rangle_{\De_{L^1}',\Boc}\\
&=(-1)^{|\beta|}\langle f, \K_{\alpha,y}\ast\partial^\beta\varphi\rangle_{\De_{L^1}',\Boc},
\end{align*}
which completes the proof.
\end{proof}

\begin{rmk}
In Alvarez et al.~\cite[Lemma 2.4]{Alvarezetal2}, a proof for the case $\alpha=0$ can be found utilizing cutoff functions and
the formula \eqref{convolutionformula}. 
\end{rmk}

\section{The Dirichlet problem}\label{sec:Dirichletproblem}

In this section we show existence of solutions to the Dirichlet problem \eqref{Dirichletproblem}
for boundary data in the weighted space $\wa\De_{L^1}'$.

\begin{dfn}\label{def:Poissonintegral}
Let $\alpha>-1$. Let $\wa$ be given by \eqref{w-alfa}, and let $f\in \wa\De_{L^1}'$.
The Poisson integral of $f$ with respect to the kernel $\K_\alpha$ is defined as the function
\[
\K_\alpha[f]:(x,y)\mapsto \K_{\alpha,y}\ast f(x),\quad (x,y)\in\Hu,
\]
where the right-hand side is the $\Es'$-convolution
of $\K_{\alpha,y}$ and $f$, and $\K_{\alpha,y}(x)=\K_{\alpha}(x,y)$
in accordance with \eqref{uyfcn}.
\end{dfn}
Henceforth, we will mostly write $\K_\alpha[f]$ only when referring to the corresponding
function on $\Hu$; the value of $\K_\alpha[f]$ at $(x,y)$ will usually still be written
as $\K_{\alpha,y}\ast f(x)$, and we will continue to write $\K_{\alpha,y}\ast f$
when discussing the map $x\mapsto \K_\alpha[f](x,y)$.
The next lemma describes the integrability properties of 
$\K_{\alpha,y}\ast f$.

\begin{lemma}\label{lem:preserves}
Let $\alpha>-1$.
If $f\in \wa\De_{L^1}'$ then
the $\Es'$-convolution $\K_{\alpha,y}\ast f$
belongs to $\wa\De_{L^1}$ for all $y>0$. In particular, the $\Es'$-convolution with
the kernel $\K_{\alpha,y}$ preserves $L^1(\wa^{-1})$ for each $y>0$.
\end{lemma}

\begin{proof}
As in the proof of Alvarez et al.~\cite[Lemma 3.1]{Alvarezetal2} we claim that it suffices to prove the implication
\begin{equation}\label{implication}
u\in\wa\De_{L^1}'\quad\Longrightarrow\quad \K_{\alpha,y}\ast u\in L^1(\wa^{-1}).
\end{equation}
Indeed, given $f\in\wa\De_{L^1}'$ we have that $\partial^\beta(\K_{\alpha,y}\ast f)= \K_{\alpha,y}\ast (\partial^\beta f)$
for any multi-index $\beta$ by Proposition \ref{prop:diffofconvolution}.
Since $\wa\De_{L^1}'$ is closed under differentiation we have $\partial^\beta f\in\wa\De_{L^1}'$,
so $\K_{\alpha,y}\ast (\partial^\beta f)$ is smooth in view of the remark on page \pageref{rmk:smoothness}
following Theorem \ref{thm:convolutionwelldefined}. Hence, if the implication \eqref{implication} holds
then
\[
\partial^\beta(\K_{\alpha,y}\ast f)\in\De_{L^1(\wa^{-1})}=\wa\De_{L^1}
\]
for any multi-index $\beta$, so $\K_{\alpha,y}\ast f\in\wa\De_{L^1}$.

Thus, let  $f\in\wa\De_{L^1}'$. By Definition \ref{def:weightedspace} we can write $f=\sum_\beta \wa \partial^\beta f_\beta$
where $f_\beta\in L^1$ and the sum is finite. We may therefore without loss of generality assume that $f=\wa\partial^\beta f_\beta$
with $f_\beta\in L^1$.
According to \eqref{convolutionformula} we then have
\begin{align}
\|\K_{\alpha,y}\ast f\|_{L^1(\wa^{-1})}
&=\int \wa^{-1}(\eta)\lvert\langle \partial^\beta f_\beta,\wa\tau_\eta \K_{\alpha,y}\rangle_{\De_{L^1}',\Boc}\rvert d\eta\\
&\le\int \wa^{-1}(\eta)\int|f_\beta(x)||\partial_x^\beta(\wa(x)\K_{\alpha,y}(\eta-x) ) |dx d\eta.
\label{firstL1normestimate}
\end{align}
By Leibniz' formula
\[
\partial_x^\beta(\wa(x)\K_{\alpha,y}(\eta-x) )=\sum_{\gamma\le\beta}\binom{\beta}{\gamma}\partial^\gamma \wa(x)
\partial^{\beta-\gamma}\K_{\alpha,y}(\eta-x).
\]
By using property \eqref{symbol} and the fact that $1+a\sim(1+a^2)^{\frac{1}{2}}$ whenever $a\ge0$, that is,
$(1+a^2)^{\frac{1}{2}}\le 1+a\le\sqrt{2}(1+a^2)^{\frac{1}{2}}$ if $a\ge0$, it is straightforward to check that
\[
|\partial_x^\beta(\wa(x)\K_{\alpha,y}(\eta-x) )|\le\sum_{\gamma\le\beta}\frac{C_{\alpha,\beta,\gamma,n}}{y^{n+|\beta-\gamma|}}
\cdot\frac{(1+|x|)^{\alpha+n+1-|\gamma|}}{\Big(1+\frac{|\eta-x|}{y} \Big)^{\alpha+n+1+|\beta-\gamma|}}.
\]
Next, note that
\[
\bigg(1+\frac{|\eta-x|}{y} \bigg)^{-\alpha-n-1-|\beta-\gamma|}\le
(1+|\eta-x|)^{-\alpha-n-1-|\beta-\gamma|}
\max(1,y^{\alpha+n+1+|\beta-\gamma|}).
\]
Hence, \eqref{firstL1normestimate} and an application of Tonelli's theorem gives
\[
\|\K_{\alpha,y}\ast f\|_{L^1(\wa^{-1})}
\le \sum_{\gamma\le\beta}C_{\alpha,\beta,\gamma,n,y}
\int |f_\beta(x)|(1+|x|)^{\alpha+n+1-|\gamma|} I(x) dx,
\]
where
\[
I(x)=\int \wa^{-1}(\eta)(1+|\eta-x|)^{-\alpha-n-1-|\beta-\gamma|}
d\eta .
\]
By Alvarez et al.~\cite[Lemma 2.8]{Alvarezetal2} we have $0<I(x)\le C_{\alpha,\beta,\gamma,n}(1+|x|)^{-\alpha-n-1}$.
Since $f_\beta\in L^1$ we conclude that
\[
\|\K_{\alpha,y}\ast f\|_{L^1(\wa^{-1})}
\le \sum_{\gamma\le\beta}C_{\alpha,\beta,\gamma,n,y}
\int |f_\beta(x)|(1+|x|)^{-|\gamma|} dx<\infty.
\]
Having proved the first statement of the lemma, the last statement follows immediately by virtue of \eqref{L1inclusions}.
\end{proof}

By virtue of \eqref{L1inclusions}, Lemma \ref{lem:preserves}
ensures that the $\Es'$-convolution $\K_{\alpha,y}\ast f$ belongs to $\wa\De_{L^1}'$
for all $y>0$ whenever $f\in\wa\De_{L^1}'$. We can therefore consider the convergence of
$\K_{\alpha,y}\ast f$ in $\wa\De_{L^1}'$ as $y\to0$.

\begin{thm}\label{thm:existence}
Let $\alpha>-1$. Let $f\in \wa\De_{L^1}'$ and set $u=\K_\alpha[f]$. Then $u_y\to f$ in $\wa\De_{L^1}'$
as $y\to0$, where $u_y(x)=u(x,y)$ for $y>0$ in accordance with \eqref{uyfcn}.
\end{thm}

\begin{proof}
We will essentially adapt the proof of Alvarez et al.~\cite[Theorem 3.6]{Alvarezetal2}.
Suppose first that we have already proved that if $g\in L^1(\wa^{-1})$ then $\K_{\alpha,y}\ast g\to g$ in $L^1(\wa^{-1})$ as $y\to0$.
Since $L^1(\wa^{-1})$ is continuously embedded in $\wa\De_{L^1}'$, it follows that
$\K_{\alpha,y}\ast g\to g$ in $\wa\De_{L^1}'$ as $y\to0$.
Now let $f\in \wa\De_{L^1}'$. By \eqref{representationformula} we can then write $f$ as a finite sum
with terms of the form $\partial^\beta f_\beta$ with $f_\beta\in L^1(\wa^{-1})$.
Recall that the operation of differentiation is continuous in $\wa\De_{L^1}'$
according to the discussion following Definition \ref{def:weightedspace}.
By Proposition \ref{prop:diffofconvolution} we have $\partial^\beta(\K_{\alpha,y}\ast f_\beta)= \K_{\alpha,y}\ast (\partial^\beta f_\beta)$
since $f_\beta\in\wa\De_{L^1}'$. This gives
\[
\K_{\alpha,y}\ast f=\sum_\beta \K_{\alpha,y}\ast (\partial^\beta f_\beta)
=\sum_\beta\partial^\beta(\K_{\alpha,y}\ast f_\beta)\to\sum_\beta\partial^\beta f_\beta=f
\]
in $\wa\De_{L^1}'$ as $y\to0$. Hence it suffices to prove that if
$f\in L^1(\wa^{-1})$ then $\K_{\alpha,y}\ast f\to f$ in $L^1(\wa^{-1})$ as $y\to0$.

Suppose therefore that $f\in L^1(\wa^{-1})$. Since $\alpha>-1$, the definition of $\wa$ ensures that
$\wa^{-1}(x)dx$ is a finite, complete, regular measure on $\R^n$, which implies that the compactly supported
continuous functions $\mathscr \Ci_0$ are dense in $ L^1(\wa^{-1})$. Given $\ve>0$ we let $g\in\Ci_0$ satisfy
\begin{equation}\label{fminusglessthanepsilon}
\|f-g\|_{L^1(\wa^{-1})}<\ve.
\end{equation}
Next, note that since $0<\wa^{-1}(x)\le 1$ for $x\in\R^n$, it follows that
\begin{align*}
\|\K_{\alpha,y}\ast g-g\|_{L^1(\wa^{-1})} & \le\iint \K_{\alpha,y}(x)|g(t-x)-g(t)|dxdt\\
&=\int \K_{\alpha,y}(x)\|\tau_xg-g\|_{L^1} dx,
\end{align*}
where the last identity follows by Tonelli's theorem since $\K_{\alpha,y}$ is nonnegative.
It is well-known that since the unweighted $L^1$
space is translation invariant, we have that $\|\tau_x g-g\|_{L^1}\to0$ as $x\to0$, see Bochner~\cite[Theorem 1.2.1]{Bochner}.
Since $\K_{\alpha,y}$ enjoys the usual properties of kernel functions, that is,  $\int \K_{\alpha,y}(x)dx=1$ for all $y>0$,
and for each $r>0$ we have
\[
\lim_{y\to0}\int_{|x|\ge r} \K_{\alpha,y}(x)dx=\lim_{y\to0} \int_{y|x|\ge r} \K_{\alpha,1}(x)dx=0,
\]
it follows that 
\begin{equation}\label{forcontinousfunctions}
\|\K_{\alpha,y}\ast g-g\|_{L^1(\wa^{-1})}\to0\quad\text{as }y\to0,
\end{equation}
see Bochner~\cite[Theorem 1.3.2]{Bochner}.
Now,
\begin{align}
\|\K_{\alpha,y}\ast f-f\|_{L^1(\wa^{-1})} & \le \|\K_{\alpha,y}\ast (f-g)\|_{L^1(\wa^{-1})}\\
& \quad +\|\K_{\alpha,y}\ast g-g\|_{L^1(\wa^{-1})}+\|g-f\|_{L^1(\wa^{-1})},
\label{convergencebound1}
\end{align}
so it only remains to estimate the first term in the right-hand side. Note that
\begin{equation}\label{convergencebound2}
\|\K_{\alpha,y}\ast (f-g)\|_{L^1(\wa^{-1})}\le\int|f(t)-g(t)|\K_{\alpha,y}\ast\wa^{-1}(t)dt
\end{equation}
by Tonelli's theorem and the fact that both $\K_{\alpha,y}$ and $\wa^{-1}$ are radial functions.

We now estimate $\K_{\alpha,y}\ast\wa^{-1}(t)$.
It is straightforward to check that the function $y\mapsto \K_{\alpha,y}(x)$ is increasing for 
$0<y<|x|/\varrho$ where $\varrho=\varrho_{\alpha,n}=\sqrt{n}/\sqrt{\alpha+1}$.
Hence, if $y<1$ and $|x|\ge 2\varrho$, then $\K_{\alpha,y}(x)\le \K_{\alpha,y+1}(x)$. This gives
\begin{align*}
\K_{\alpha,y}\ast\wa^{-1}(t) & \le \int_{|x|\ge2\varrho} \K_{\alpha,y+1}(x)\wa^{-1}(t-x)dx\\
&\quad +\int_{|x|<2\varrho} \K_{\alpha,y}(x)\wa^{-1}(t-x)dx=I_1(t)+I_2(t).
\end{align*}
Below we use the fact that $1+a\sim(1+a^2)^{\frac{1}{2}}$ whenever $a\ge0$, and we let $C_{\alpha,n}$
denote a constant, depending on $\alpha$ and $n$, the value of which is permitted to change between occurrences. We first estimate $I_1(t)$.
For $0<y<1$ we have $ \K_{\alpha,y+1}(x)\le 2^{\alpha+1}\K_{\alpha,1}(x)$, which gives
\[
I_1(t)\le C_{\alpha,n}\int_{|x|\ge2\varrho}(1+|x|)^{-(\alpha+n+1)/2}
(1+|t-x|)^{-(\alpha+n+1)/2}dx.
\]
Since $\alpha>-1$, it follows by an application of Alvarez et al.~\cite[Lemma 2.8]{Alvarezetal2} that the right-hand side is finite and bounded by a constant $C_{\alpha,n}$
multiplied by $\wa^{-1}(t)$. To estimate $I_2(t)$ we apply Peetre's inequality \eqref{Peetre}
to $\wa^{-1}(t-x)$ which gives
\begin{align*}
I_2(t)&\le C_{\alpha,n}\wa^{-1}(t)\int_{|x|<2\varrho} \K_{\alpha,y}(x)\wa(x)dx\\
&\le C_{\alpha,n}\wa(2\varrho)\wa^{-1}(t)\int \K_{\alpha,y}(x)dx=C_{\alpha,n}\wa^{-1}(t).
\end{align*}
Combining the estimates for $I_1(t)$ and $I_2(t)$ we thus have $0<\K_{\alpha,y}\ast\wa^{-1}(t)\le C_{\alpha,n}\wa^{-1}(t)$.
By virtue of \eqref{convergencebound2} we find that $\|\K_{\alpha,y}\ast (f-g)\|_{L^1(\wa^{-1})}\le C_{\alpha,n}\|f-g\|_{L^1(\wa^{-1})}$.
In view of \eqref{fminusglessthanepsilon}--\eqref{convergencebound1} this implies that
$\|\K_{\alpha,y}\ast f-f\|_{L^1(\wa^{-1})}\le (C_{\alpha,n}+2)\ve$
for any sufficiently small $y>0$, which completes the proof.
\end{proof}

\begin{cor}\label{cor:existence}
Let $\alpha>-1$. Let $f\in \wa\De_{L^1}'$. Then $\K_\alpha[f]$
is a solution to the Dirichlet problem \eqref{Dirichletproblem}.
\end{cor}

\begin{proof}
Set $u=\K_\alpha[f]$ and note that $\wa\De_{L^1}'\subset\Es'$ with continuous inclusion.
By Theorem \ref{thm:existence} we thus have $u_y\to f$ in $\Es'$ as $y\to0$,
where $u_y(x)=u(x,y)$ for $y>0$ in accordance with \eqref{uyfcn}.
In view of the remark on page \pageref{rmk:smoothness},
we have $D_\alpha u=0$ in $\Hu$, which completes the proof.
\end{proof}

Compared to the harmonic case $\alpha=0$,
the proof of Theorem \ref{thm:existence} is complicated by the fact that when $\alpha\ne0$,
the family $\{\K_{\alpha,y}\}_{y>0}$ does not in general satisfy the semi-group
property $\K_{0,y_1+y_2}=\K_{0,y_1}\ast \K_{0,y_2}$ enjoyed by the classical Poisson kernel
$\K_{0}$. (This would have allowed for an easier estimation of
$\K_{\alpha,y}\ast\wa^{-1}(t)$.) In fact, in the context of hyperbolic Brownian motion it is a point of interest
to determine, for each fixed $\eta>0$, the function $\mathscr G_\alpha\equiv
\mathscr G_\alpha(\eta)$ such that $\mathscr G_\alpha:(x,y)\mapsto \mathscr G_\alpha(x,y)$ satisfies
\begin{equation}\label{eq:hittingrandomvariable}
\K_{\alpha,y}=\mathscr G_{\alpha,y}\ast \K_{\alpha,\eta},\quad y>\eta,
\end{equation}
where $\mathscr G_{\alpha,y}(x)=\mathscr G_{\alpha}(x,y)$ in accordance with \eqref{uyfcn}.
$\mathscr G_\alpha$ is then the kernel function for the Dirichlet problem $D_\alpha u=0$
in the half space $y>\eta$ with boundary conditions given on the hyperplane $y=\eta$.
In other words, it is the probability density function of the measure of probability
that the process \eqref{HyperbolicBrownianMotion} with $\mu=\frac{\alpha}{2}+\frac{1}{2}$ and
starting at $(0,y)$, $y>\eta$, hits a portion of the boundary $y=\eta$.
(For this probability measure, the terminology {\it hitting distribution}
is commonly used.) We would thus like to solve \eqref{eq:hittingrandomvariable},
a priori interpreted in the distributional sense by means of the $\Es'$-convolution,
and find a solution in $\wa\De'_{L^1}$ to justify the equation. Since the
$\Es'$-convolution satisfies the Fourier exchange formula, a necessary condition is that
$\widehat{\K_{\alpha,y}}=\widehat{\mathscr G_{\alpha,y}}\widehat{\K_{\alpha,\eta}}$.
Since the Fourier transform of $\K_{\alpha,y}$ is nonvanishing for each $y>0$, this is equivalent to
\begin{equation}\label{eq:characteristicfunction}
\widehat{\mathscr G_{\alpha,y}}(\xi)=
\frac{\widehat{\K_{\alpha,y}}(\xi)}{\widehat{\K_{\alpha,\eta}}(\xi)}=\left(\frac{y}{\eta}\right)^{(\alpha+1)/2}
\frac{K_{(\alpha+1)/2}(y\lvert\xi\rvert)}{K_{(\alpha+1)/2}(\eta\lvert\xi\rvert)},\quad y>\eta,
\end{equation}
where the second formula follows from Theorem \ref{thm:Ftransform}. As before, $K_{(\alpha+1)/2}$ denotes the
modified Bessel function of the third kind of order $\frac{\alpha}{2}+\frac{1}{2}$.
For $n=1$, this formula appears for example in Baldi et al.~\cite[Section 4]{BaldiEtAl2}
(choosing $y=1$ and $\nu=\frac{\alpha}{2}+\frac{1}{2}$).
An equivalent formula also appears in
Byczkowski et al.~\cite[Theorem 2.1]{ByczkowskiEtAl} in the special case $\alpha=n-1$.
In both instances the proofs involve probabilistic methods.

Using Corollary \ref{cor:Ftransform} together with the first identity in \eqref{eq:characteristicfunction} we have
\begin{equation}
\widehat{\mathscr G_{\alpha,y}}(\xi)=\left(\frac{y}{\eta}\right)^{\alpha+1}\frac{\int_1^\infty 
e^{-y\lvert\xi\rvert t}
(t^2-1)^{\alpha/2} dt}{\int_1^\infty 
e^{-\eta\lvert\xi\rvert t} (t^2-1)^{\alpha/2} dt}\le\left(\frac{y}{\eta}\right)^{\alpha+1}e^{-(y-\eta)\lvert\xi\rvert},\quad y>\eta,
\end{equation}
which shows that $\widehat{\mathscr G_{\alpha,y}}$ is rapidly decreasing and belongs to $L^2$.
By means of the Fourier inversion formula, this gives an element
$\mathscr G_{\alpha,y}\in\Ci^\infty$, uniquely defined in $L^2$.
Since $\K_{\alpha,\eta}\in L^1$ by Lemma \ref{L1means}, the convolution
$\mathscr G_{\alpha,y}\ast\K_{\alpha,\eta}$ is well defined in the usual sense; moreover,
it belongs to $L^2$ by Young's inequality and the Fourier exchange formula holds, so
\begin{equation}\label{eq:exchangeformulaconvolution}
\widehat{\mathscr G_{\alpha,y}\ast\K_{\alpha,\eta}}=\widehat{\mathscr G_{\alpha,y}}\widehat{\K_{\alpha,\eta}}=\widehat{\K_{\alpha,y}}.
\end{equation}
It follows that $\mathscr G_\alpha$
is a solution to \eqref{eq:hittingrandomvariable}. Indeed, it is straightforward to check
that the translation $\tau_{h}\K_{\alpha,\eta}$ tends to $\K_{\alpha,\eta}$ in $L^2$ as $h\to0$,
and that $\mathscr G_{\alpha,y}\ast\K_{\alpha,\eta}$ therefore is continuous.
By Parseval's formula and \eqref{eq:exchangeformulaconvolution} we have
\begin{equation}
\int \K_{\alpha,y}(x)\varphi(x)dx=
\int\widehat{\K_{\alpha,y}}(\xi)\hat\varphi(-\xi)d\xi
=\int\mathscr G_{\alpha,y}\ast\K_{\alpha,\eta}(x)\varphi(x)dx
\end{equation}
for all $\varphi\in\Ci_0^\infty$, which implies that $\K_{\alpha,y}=\mathscr G_{\alpha,y}\ast\K_{\alpha,\eta}$,
see H{\"o}rmander~\cite[Theorem 1.2.4]{Hormander1}. 

Note that the Fourier transform of $\mathscr G_{\alpha,y}$
is radial by \eqref{eq:characteristicfunction}. In view of the discussion surrounding
\eqref{eq:steinandweiss}, the Fourier inversion formula can
therefore be used to obtain the representation formula
\begin{align}
\mathscr G_{\alpha,y}(x)
=\frac{\lvert x\rvert^{(2-n)/2}}{(2\pi)^{n/2}}\left(\frac{y}{\eta}\right)^{(\alpha+1)/2}\int_0^\infty
\frac{K_{(\alpha+1)/2}(ys)}{K_{(\alpha+1)/2}(\eta s)}s^{n/2}J_{(n-2)/2}(\lvert x\rvert s)ds.
\end{align}
This formula appears in
Byczkowski et al.~\cite[Theorem 2.2]{ByczkowskiEtAl} in the special case $\alpha=n-1$.
Writing $\mathscr G_{\alpha,y}(\eta)$ for the function appearing above,
they also show that when $\alpha=n-1$, the family $\{\mathscr G_{\alpha,y}(\eta)\}_{0<\eta<y}$
satisfies the semi-group property
\[
\mathscr G_{\alpha,y}(\eta_1)=\mathscr G_{\alpha,y}(\eta_2)\ast \mathscr G_{\alpha,\eta_2}(\eta_1),
\quad 0<\eta_1<\eta_2<y.
\]
Clearly, this continues to hold for arbitrary $\alpha>-1$; in fact,
it is an immediate consequence of the first identity in \eqref{eq:characteristicfunction}
in view of the previous discussion.
For completeness we collect these observations in the following proposition.

\begin{prop}\label{prop:hitting}
Let $\alpha>-1$ and let $\K_\alpha$ be given by Definition \ref{def:Poissonkernel}. For each $\eta>0$ there is a uniquely defined function
$\mathscr G_\alpha(\eta)$ 
satisfying
$\K_{\alpha,y}=\mathscr G_{\alpha,y}(\eta)\ast \K_{\alpha,\eta}$ for $y>\eta$. $\mathscr G_\alpha(\eta)$
can be represented by
\begin{align}
\mathscr G_{\alpha}(\eta):(x,y)\mapsto 
\frac{\lvert x\rvert^{(2-n)/2}}{(2\pi)^{n/2}}\left(\frac{y}{\eta}\right)^{(\alpha+1)/2}\int_0^\infty
\frac{K_{(\alpha+1)/2}(ys)}{K_{(\alpha+1)/2}(\eta s)}s^{n/2}J_{(n-2)/2}(\lvert x\rvert s)ds,
\end{align}
where $J_\nu$ denotes the Bessel function of the first kind of order $\nu$.
Moreover, the Fourier transform of $\mathscr G_{\alpha,y}(\eta)$ is given by \eqref{eq:characteristicfunction},
and the family $\{\mathscr G_{\alpha,y}(\eta)\}_{0<\eta<y}$
satisfies the semi-group property
\[
\mathscr G_{\alpha,y}(\eta_1)=\mathscr G_{\alpha,y}(\eta_2)\ast \mathscr G_{\alpha,\eta_2}(\eta_1),
\quad 0<\eta_1<\eta_2<y.
\]
\end{prop}

\section{Asymptotic behavior of the Poisson integral}\label{sec:asymptoticgrowth}

In this section we investigate asymptotic growth behavior of the Poisson
integral $\K_{\alpha}[f]$ when $f\in\wa\De'_{L^1}$.
We shall obtain growth estimates comparable to those
satisfied by the classical ($\alpha=0$) Poisson integral of $p$-summable functions proved by
Siegel and Talvila~\cite{SiegelTalvila}. We begin with a lemma, which is essentially
just a version of~\cite[Theorem 2.1]{SiegelTalvila} in the presence of additional parameters,
proved using similar techniques. We shall write $S_r$ to denote
the set
\[
S_r=\Hu\cap\{(x,y)\in\R^{n+1}:x^2+y^2=r^2\},\quad r>0.
\]

\begin{lemma}\label{lem:growthestimate}
Let $\alpha>-1$. For each $k\in\N$, define $L_{\alpha,k}:\Hu\to\R$ by
\[
L_{\alpha,k}(x,y)=\frac{y^{\alpha+1}}{(x^2+y^2)^{(\alpha+n+1+k)/2}}.
\]
$($Thus, modulo a scaling factor we have $\K_{\alpha}=L_{\alpha,0}$.$)$
If $\mu\le \alpha+n+1$ and $f\in L^1(w^{-\mu })$ then
\[
\int L_{\alpha,k}(x-\eta,y)\lvert f(\eta)\rvert d\eta
\le C_{\mu}I(r)r^{\mu-n-k}\sec^{n+k}\vartheta,\quad (x,y)\in S_r,
\]
where $I(r)\to0$ as $r\to\infty$, and $\vartheta\in[0,\pi/2)$ is the angle defined by
$y=r\cos\vartheta$ and $\lvert x\rvert=r\sin\vartheta$ for $(x,y)\in S_r$.
\end{lemma}

\begin{proof}
Introduce the angle $\omega$ defined by $\langle x,\eta\rangle=\lvert x\rvert\lvert\eta\rvert\cos\omega$
for $x$ and $\eta$ in $\R^n$. By the definition of the angle $\vartheta$
we then have $\langle x,\eta\rangle=\lvert\eta\rvert r\sin\vartheta\cos\omega$, which gives
\begin{align*}
\lvert x-\eta\rvert^2+y^2&=(\lvert\eta\rvert-r)^2+2\lvert\eta\rvert r-2\langle x,\eta\rangle\\
&=(1-\sin\vartheta\cos\omega)\bigg(\frac{(\lvert\eta\rvert-r)^2}{1-\sin\vartheta\cos\omega}+2\lvert\eta\rvert r\bigg).
\end{align*}
It is straightforward to check that
\begin{equation}\label{lowerupperbound}
\frac{\cos^2\vartheta}{1+\sin\vartheta}\le 1-\sin\vartheta\cos\omega\le2,
\end{equation}
and using these bounds we obtain the estimate
\[
\lvert x-\eta\rvert^2+y^2\ge\frac{\cos^2\vartheta}{1+\sin\vartheta}
\bigg(\frac{(\lvert\eta\rvert-r)^2}{2}+2\lvert\eta\rvert r\bigg)
=\frac{\cos^2\vartheta}{1+\sin\vartheta}\cdot\frac{(\lvert\eta\rvert +r)^2}{2}.
\]
Hence, for $\mu\ge0$ we have
\begin{equation}\label{smartbound}
(\lvert x-\eta\rvert^2+y^2)^{-\mu/2}\le\frac{(1+\sin\vartheta)^{\mu/2}}{\cos^{\mu}\vartheta}
\cdot\frac{2^{\mu/2}}{(\lvert\eta\rvert+r)^{\mu}}
\le\frac{2^{\mu}}{\cos^{\mu}\vartheta}(\eta^2+r^2)^{-\mu/2},
\end{equation}
where we in the last inequality also use the fact that $(a^2+b^2)^{\frac{1}{2}}\le a+b$ when $a$ and $b$ are positive real numbers.
Now write
\begin{equation}
\int L_{\alpha,k}(x-\eta,y)\lvert f(\eta)\rvert d\eta=
\int \frac{\lvert f(\eta)\rvert}{(\lvert x-\eta\rvert^2+y^2)^{\mu/2}}
\cdot\frac{y^{\alpha+1}d\eta}{(|x-\eta|^2+y^2)
^{(\alpha+n+1+k-\mu)/2}}
\end{equation}
and note that
\begin{equation}\label{eq:supinsteadofLq}
y^{\alpha+1}\sup_{\eta}{(\lvert x-\eta\rvert^2+y^2)^{-(\alpha+n+1+k-\mu)/2}}
=y^{\mu-n-k}
\end{equation}
for $k\in\N$ when $\mu\le\alpha+n+1$.
Hence, by virtue of \eqref{smartbound} we find  that
\begin{equation}\label{usingholder}
\int L_{\alpha,k}(x-\eta,y)\lvert f(\eta)\rvert d\eta
\le 2^\mu I(r) y^{\mu-n-k}\sec^{\mu}{\vartheta},
\end{equation}
where
\begin{equation}\label{Ioneofr}
I(r)=\int\frac{\lvert f(\eta)\rvert}{(\eta^2+r^2)^{\mu/2}}d\eta\to0\quad
\text{as }r\to\infty
\end{equation}
by the dominated convergence theorem since $I(r)\le \|f\|_{L^1(w^{-\mu })}$ when $r=(x^2+y^2)^{\frac{1}{2}}\ge 1$.
Recalling that for $(x,y)\in S_r$ we have $y=r\cos\vartheta$ we obtain
\[
\int L_{\alpha,k}(x-\eta,y)\lvert f(\eta)\rvert d\eta\le
2^\mu I(r) r^{\mu-n-k}\sec^{n+k}\vartheta,
\]
which yields the result.
\end{proof}

Before using Lemma \ref{lem:growthestimate} to obtain
growth estimates of the 
Poisson integral $\K_\alpha[f]$
for general $f\in\wa\De'_{L^1}$, we mention that
an application of the lemma with $k=0$ and $\mu=\alpha+n+1$
shows that
if $f\in L^1(\wa^{-1})$ then
\begin{equation}\label{eq:orderrelation}
\sup_{S_r}{\lvert \K_{\alpha,y}\ast f(x)\cos^n\vartheta\rvert}=o(r^{\alpha+1})
\quad\text{as }r\to\infty.
\end{equation}
When $\alpha=0$ we recover the corresponding result of Siegel and Talvila~\cite[Corollary 2.1]{SiegelTalvila}
for the usual Poisson integral of elements in $L^1(w^{-(n+1)})$. 
In analogy, the order relation \eqref{eq:orderrelation} is sharp in the sense that the exponents cannot in general be decreased.
To prove this, the arguments used for $\alpha=0$ by
Siegel and Talvila~\cite[p.~576]{SiegelTalvila} 
are adapted to handle the full parameter range $\alpha>-1$.

Let $\hat e_1$ be the unit vector along the $x_1$ axis.
Let $f_k$, $a_k$ and $\varrho_k$ be positive
real numbers such that $\varrho_k<1$, $a_k\to\infty$ as $k\to\infty$
and the balls $B_{\varrho_k}(a_k\hat e_1)\subset\R^n$ with center at $a_k\hat e_1$
and radius $\varrho_k$ are disjoint. Define a continuous function
$f$ vanishing outside these balls by
\[
f(x)=
\begin{cases}
f_k(1-\frac{1}{\varrho_k}\lvert x-a_k\hat e_1\rvert), & x\in B_{\varrho_k}(a_k\hat e_1),\\
0 & \text{otherwise}.
\end{cases}
\]
It is straightforward to check that $f\in L^1(\wa^{-1})$ if and only if
\begin{equation}\label{eq:seriesconverges}
\sum_k f_k\frac{\varrho_k^n}{a_k^{\alpha+n+1}}<\infty,
\end{equation}
and that if $u(x,y)=\K_{\alpha}[f](x,y)$, then $u$ can be written as
a superposition
\[
u(x,y)=\sum_k f_k\int_{B_1(0)}(1-\lvert \eta\rvert) \K_{\alpha,y}\bigg(\frac{x-a_k\hat e_1}{\varrho_k}-\eta,
\frac{y}{\varrho_k}\bigg) d\eta
\]
of translates of the function $\tilde u(x,y)=\K_{\alpha,y}\ast\max{(0,1-\lvert{\hdot}\rvert)}(x)$.
Here $B_1(0)$ is the unit ball in $\R^n$. Note that $D_\alpha \tilde u=0$ in $\Hu$
and $\tilde u(x,0)=\max{(0,1-\lvert x\rvert)}$. Since $\tilde u\ge0$ we thus have
\begin{equation}\label{eq:lowerboundforeachk}
u(x,y)\ge f_k\tilde u\bigg(\frac{x-a_k\hat e_1}{\varrho_k},
\frac{y}{\varrho_k}\bigg),\quad k\ge1.
\end{equation}
We now claim that if $\beta+\gamma<\alpha+n+1$, or $\beta+\gamma=\alpha+n+1$
but $\gamma<n$, then $r^{-\beta}u(x,y)\cos^\gamma\vartheta$ does not tend
to zero along the sequence $(x^{(k)},y^{(k)})=(a_k\hat e_1,\varrho_k)$
for appropriate choices of the numbers $f_k$, $a_k$ and $\varrho_k$.
Indeed, if $\beta+\gamma<\alpha+n+1$, set $a_k=e^k$, $f_k=e^{k(\alpha+n+1)}$
and $\varrho_k=k^{-2}$. Then the series \eqref{eq:seriesconverges}
is easily seen to be convergent, while
\[
\frac{(y^{(k)})^\gamma u(x^{(k)},y^{(k)})}{((x^{(k)})^2+(y^{(k)})^2)^{(\beta+\gamma)/2}}
\ge\frac{\varrho_k^\gamma f_k\tilde u(0,1)}{(a_k^2+\varrho_k^2)^{(\beta+\gamma)/2}}
=\frac{k^{-2\gamma}e^{k(\alpha+n+1)}\tilde u(0,1)}{(e^{2k}+k^{-4})^{(\beta+\gamma)/2}}
\]
does not tend to zero as $k\to\infty$ since $\beta+\gamma<\alpha+n+1$ and
\[
\tilde u(0,1)=\frac{\Gamma((\alpha+n+1)/2)}{\Gamma((\alpha+1)/2)\pi^{n/2}}
\int_{B_1(0)}\frac{(1-\lvert\eta\rvert)}{(1+\eta^2)}d\eta\ne 0.
\]
If instead $\beta+\gamma=\alpha+n+1$,
but $\gamma<n$, let $\ve=n-\gamma>0$ and set
$a_k=e^k$, $f_k=e^{k(\alpha+n+1)}k^{\gamma(1+\ve)/\ve}$
and $\varrho_k=k^{-(1+\ve)/\ve}$. Then the left-hand side of
\eqref{eq:seriesconverges} is equal to $\sum_k k^{-1-\ve}<\infty$.
However, $\varrho_k^\gamma f_k=e^{k(\alpha+n+1)}$, so
\[
\frac{(y^{(k)})^\gamma u(x^{(k)},y^{(k)})}{((x^{(k)})^2+(y^{(k)})^2)^{(\beta+\gamma)/2}}
\ge\frac{e^{k(\alpha+n+1)}\tilde u(0,1)}{(e^{2k}+k^{-2(1+\ve)/\ve})^{(\alpha+n+1)/2}}
\]
does not tend to zero as $k\to\infty$. Thus,
the order relation \eqref{eq:orderrelation} is sharp.

\begin{thm}\label{thm:growthestimate}
Let $\alpha>-1$. Let $f\in \wa\De_{L^1}'$. Then there exists a nonnegative integer $m$ depending only on
the distribution $f$ such that
the Poisson integral $\K_\alpha[f]$
satisfies
\begin{equation}\label{eq:asymptoticgrowth}
\sup_{S_r}{\lvert \K_{\alpha,y}\ast f(x)\cos^{n+m}\vartheta\rvert}=o(r^{\alpha+1})\quad\text{as }r\to\infty,
\end{equation}
where $\vartheta\in[0,\pi/2)$ is the angle defined by
$y=r\cos\vartheta$ and $\lvert x\rvert=r\sin\vartheta$ for $(x,y)\in S_r$.
\end{thm}

\begin{proof}
By the representation formula \eqref{representationformula} we have $f=\sum_{\lvert\beta\rvert\le m} \partial^\beta f_\beta$
where $f_\beta\in L^1(\wa^{-1})$ and $m\in\N$. Note that
\[
\lvert\partial_\eta^\beta \K_{\alpha,y}(x-\eta)\rvert\le C_{\beta,\alpha,n}
\frac{y^{\alpha+1}}{(\lvert x-\eta\rvert^2+y^2)^{(\alpha+n+1+\lvert\beta\rvert)/2}}
=C_{\beta,\alpha,n} L_{\alpha,\lvert\beta\rvert}(x-\eta,y),
\]
where $L_{\alpha,k}$ is the function defined in Lemma \ref{lem:growthestimate} for $k\in\N$.
Thus
\begin{align*}
\K_{\alpha,y}\ast f(x)&
=\sum_{\lvert\beta\rvert\le m}\langle \wa^{-1}\partial^\beta f_\beta,\wa\tau_x \K_{\alpha,y}\rangle_{\De_{L^1}',\Boc}\\
&=\sum_{\lvert\beta\rvert\le m}(-1)^{\lvert\beta\rvert}\langle f_\beta,\partial^\beta\tau_x \K_{\alpha,y}\rangle_{\De_{L^1}',\Boc}
\end{align*}
where each term $\langle f_\beta,\partial^\beta\tau_x \K_{\alpha,y}\rangle_{\De_{L^1}',\Boc}$
in the right-hand side can be identified with the corresponding integral $\int f_\beta(\eta)\partial_\eta^\beta \K_{\alpha,y}(x-\eta)d\eta$
in view of Lemma \ref{lem:growthestimate}. Applying the lemma with $\mu=\alpha+n+1$ gives the estimate
\begin{align*}
\lvert \K_{\alpha,y}\ast f(x)\rvert &
\le \sum_{\lvert\beta\rvert\le m}C_\beta
r^{\alpha+1-\lvert\beta\rvert}\sec^{n+\lvert\beta\rvert}\vartheta I_\beta(r)\\
&\le C_m r^{\alpha+1}\sec^{n+m}\vartheta\sum_{k=0}^m r^{-k}\sec^{k-m}\vartheta R_k(r),
\end{align*}
where $R_k(r)=\sum_{\lvert\beta\rvert=k}I_\beta(r)\to0$ as $r\to\infty$. Since $\lvert r^{-k}\sec^{k-m}\vartheta\rvert\le 1$
for $0\le k\le m$ and $r\ge1$, this completes the proof.
\end{proof}

As a final note, we shall briefly discuss the question of uniqueness of solutions to \eqref{Dirichletproblem}.
The result by Huber~\cite{HuberUnique} included above as Theorem \ref{thm:Huberuniqueness} makes evident
that a growth condition at infinity
is needed in order to have uniqueness of solutions for the Dirichlet problem \eqref{Dirichletproblem}.
In fact, the function
\begin{equation}\label{exfcn}
u(x,y)=y^{\alpha+1},\quad (x,y)\in\Hu,
\end{equation}
satisfies $D_\alpha u=0$ in $\Hu$ and vanishes on the boundary.
In analogy with the unweighted case, solutions to \eqref{Dirichletproblem}
satisfy the following principle of Phragm{\'e}n-Lindel{\"o}f type, also due to Huber~\cite{HuberPL}.
The result is stated verbatim but using our choice of notation.

\begin{thm}[A. Huber]\label{thm:HuberPL}
Let $u$ be a solution of $D_\alpha u=0$ $(\alpha>-1)$, defined in $\Hu$ and satisfying at
the boundary
\begin{flalign*}
 && \limsup_{(x,y)\to (x_0,0)} u(x,y) \le 0  && \llap{$((x,y)\in\Hu;\  x_0\in\R^n)$.}
\end{flalign*}
If follows that
\begin{enumerate}
\item[$(\mathrm{a})$] the limit $\varrho=\lim_{r\to\infty}m(r)/r^{\alpha+1}$, where $m(r)=\sup_{(x,y)\in S_r}u(x,y)$,
always exists $($finite or infinite$)$,
\item[$(\mathrm{b})$] $\varrho\ge 0$,
\item[$(\mathrm{c})$] $u\le \varrho y^{\alpha+1}$ holds throughout $\Hu$,
\item[$(\mathrm{d})$] if in $\mathrm{(c)}$ the equality is attained in at least one point of
$\Hu$, then $u\equiv \varrho y^{\alpha+1}$.
\end{enumerate}
\end{thm}

By regularization we immediately obtain the following analog
for boundary values interpreted in the sense of \eqref{Dirichletproblem}.

\begin{cor}\label{Dalfauniqueness.distribution}
Let $\alpha>-1$. Let $u$ be a solution to the equation $D_\alpha u=0$ in $\Hu$ such that 
$\sup_{S_r}\!u(x,y)=o(r^{\alpha+1})$ as $r\to\infty$  and $\lim_{y\to 0}u_y=0$ in $\Es'$. 
Then $u=0$ in $\Hu$.
\end{cor}

\begin{proof}
Let $\psi\in \Ci_0^\infty$ be a 
compactly supported test function on $\R^n$ and consider the regularization
$$
u_\psi(x,y)=\int u(x-\eta,y)\psi(\eta)d\eta, \quad y>0,  
$$
of $u$. Thus
$u_\psi(x,y)=u_y\ast\psi(x)$ so $D_\alpha u_\psi=0$ in $\Hu$. If we regard
$u_y$ as a distribution in $\Es'$ arising from the map $\eta\mapsto u_y(\eta)$,
then $u_\psi$ can also be regarded as the function $u_\psi(x,y)=\langle u_y,\psi(x-{\hdot})\rangle$
obtained by letting $u_y$ act on $\eta\mapsto\psi(x-\eta)$. Since translation
is continuous on $\Ci_0^\infty$ it thus follows that
\[
\limsup_{(x,y)\to (x_0,0)}u_\psi(x,y)=\limsup_{(x,y)\to (x_0,0)}\langle u_y,\psi(x-{\hdot})\rangle=0
\]
for all $x_0\in\R^n$ since $u_y\to 0$ in $\Es'$, see H{\"o}rmander~\cite[Theorem 2.1.8]{Hormander1}.
Moreover, 
the growth assumption on $u$ together with the fact that $\psi$ is compactly supported
implies that $u_\psi$ satisfies the same growth condition.
In fact, by assumption we can
for every $\ve>0$ find $r_\ve$ such that
\begin{equation}\label{eq:asymptoticassumption}
(\xi^2+\eta^2)^{-(\alpha+1)/2}u(\xi,\eta)<\frac{\ve}{2^{\alpha+1}\lVert\psi\rVert_{L^1}},
\quad \xi^2+\eta^2\ge r_\ve.
\end{equation}
If $\supp{\psi}\subset(-R,R)$, it follows that for any $r\ge r_\ve+R$ we have
$(x-t)^2+y^2\ge r_\ve^2$ for $(x,y)\in S_r$. By \eqref{eq:asymptoticassumption}
this gives
\begin{equation}
u_\psi(x,y)<\frac{\ve}{2^{\alpha+1}\lVert\psi\rVert_{L^1}}\int_{-R}^R((x-t)^2+y^2)^{(\alpha+1)/2}\lvert\psi(t)\rvert dt
\le\ve r^{\alpha+1}
\end{equation}
for $(x,y)\in S_r$, so $\sup_{S_r}\!u_\psi(x,y)=o(r^{\alpha+1})$ as $r\to\infty$.
Hence, Theorem \ref{thm:HuberPL}
implies that $u_\psi\le0$ in $\Hu$. A repetition of the arguments applied to $-u$
shows that $u_\psi=0$ in $\Hu$. Varying $\psi\in \Ci_0^\infty$ we conclude
that $u(x,y)=0$ for all $(x,y)\in\Hu$.
\end{proof}

Unfortunately, the growth conditions imposed in 
Corollary \ref{Dalfauniqueness.distribution} are not compatible with those satisfied
by the Poisson integral $\K_{\alpha}[f]$ of elements $f\in\wa\De_{L^1}'$; in fact
$L^1(\wa^{-1})\subset\wa\De'_{L^1}$ and the order relation \eqref{eq:orderrelation} is sharp
for $f\in L^1(\wa^{-1})$. Hence, stronger results are needed if we are to conclude
that $u=\K_{\alpha}[f]$ is the unique solution to the Dirichlet problem \eqref{Dirichletproblem}
under appropriate growth constraints.
In the harmonic case $\alpha=0$, a uniqueness result
for the classical Dirichlet problem with continuous boundary data
has been obtained by Siegel and Talvila~\cite[Theorem 3.1]{SiegelTalvila}
assuming a growth condition of type \eqref{eq:orderrelation}.
For distributional boundary data, Alvarez, Guzm{\'a}n-Partida and P{\'e}rez-Esteva~\cite[Theorem 4.1]{Alvarezetal2}
provide conditions under which functions harmonic in $\Hu$ may be
represented as Poisson integrals of the data, modulo constant multiples of the
nontrivial solution $(x,y)\mapsto y$.

\section*{Acknowledgement}
I am deeply grateful to Anders Olofsson for many helpful discussions on the subject.
I also wish to express my gratitude to Yoshinori Morimoto and Kyoto University for their
hospitality. The research was supported in part
by JSPS Kakenhi Grant No.~24-02782,
Japan Society for the Promotion of Science.


\bibliographystyle{amsplain}
\bibliography{referenser}

\end{document}